\documentclass[11pt,leqno]{article}
\usepackage{amssymb, amscd, amsmath, amsthm}

\allowdisplaybreaks

\def\ve{\varepsilon}

\def\mod{\,\text{\rm mod}\;}
\def\beq{\begin{equation}}
\def\eeq{\end{equation}}
\def\cond{\text{\rm cond}\,}
\def\mfS{{\mathfrak S}}
\def\mcR{{\mathcal R}}
\newtheorem{Thm}{Theorem}
\newtheorem{Lem}{Lemma}
\newtheorem{Cor}{Corollary}
\newtheorem*{Corol}{Corollary}
\begin{document}

\title{A new explicit formula in the additive theory of primes with applications II.
The exceptional set in Goldbach's problem}

\author{by\\
J\'anos Pintz\thanks{Supported by ERC-AdG.~321104 and National Research Development and Innovation Office, NKFIH, K~119528.}}

\date{Dedicated to the 60\textsuperscript{th} birthday of Sz. Gy. R\'ev\'esz}

\numberwithin{equation}{section}
%%\numberwithin{Lem}{section}

\maketitle

\section{Introduction}
\label{sec:1}

The first non-trivial, although conditional, estimate for the exceptional set $(\mathcal P$ denotes the set of primes)
\begin{equation}
\mathcal E = \{2\mid n;\  n \neq p + p',\ p, p' \in \mathcal P\}
\label{eq:1.1}
\end{equation}
in Goldbach's problem was achieved by Hardy and Littlewood \cite{HL} in 1924.
They showed under the Generalized Riemann Hypothesis (GRH) the estimate
\beq
E(X) = \{n \leq X; \ n \in \mathcal E\} \ll_\ve X^{1/2 + \ve}
\label{eq:1.2}
\eeq
for any $\ve > 0$. This result is apart from the substitution of $X^\ve$ by $\log^c X$ by Goldston \cite{Go} even today the best conditional result on~GRH.

The basic result that almost all even integers are Goldbach numbers, that is, can be represented as the sum of two primes, was proved by the aid of Vinogradov's method \cite{Vin} in 1937/38 simultaneously and independently by Van der Corput \cite{VdC}, \v Cudakov \cite{Cud} and Estermann \cite{Est}.
They showed
\beq
E(X) \ll_A \frac{X}{(\log X)^A} \quad \text{ for any }\ A > 0.
\label{eq:1.2masodik}
\eeq
It is easy to see that \eqref{eq:1.2masodik}, even any estimate of type
\beq
E(2N) < \pi(2N) - 1 \quad \text{ for } \ N > N
\label{eq:1.3}
\eeq
implies Vinogradov's three primes theorem \cite{Vin} that every sufficiently large odd integer can be written as the sum of three primes.

The result \eqref{eq:1.2masodik} held the record for 35 years, when Vaughan \cite{Vau} improved it in 1972 to
\beq
E(X) \ll X \exp \bigl( - c \sqrt{\log X}\bigr).
\label{eq:1.4}
\eeq
The breakthrough came just 3 years later when Montgomery and Vaughan \cite{MV} showed that the estimate
\beq
E(X) \ll_\delta X^{1 - \delta}
\label{eq:1.5}
\eeq
holds with an unspecified but explicitly calculable value $\delta > 0$.

The problem, to show \eqref{eq:1.5} with a not too small explicit value of $\delta$ turned out to be very difficult.
It was shown in 1989 by J. R. Chen and M. Liu \cite{CL} that \eqref{eq:1.5} holds with $\delta = 0.05$, ten years later by H. Z. Li \cite{Li1} that also $\delta = 0.079$ is admissible.
This was improved by him \cite{Li2} to
\beq
E(X) < X^{0.914}
\label{eq:1.6}
\eeq
for any $X > X_0$, ineffective constant.
Finally in 2010 Lu \cite{Lu} succeeded to show
\beq
\label{eq:uj1.8}
E(X) < X^{0{.}879}
\eeq
for $X > X_2$ ineffective constant.

The present work will be devoted to the proof of the following result.

\begin{Thm}
\label{th:1}
There is an ineffective constant $X_2$ such that for $X > X_2$
\beq
E(X) < X^{0.72}.
\label{eq:1.8}
\eeq
%%where
\end{Thm}

The seemingly moderate size of the present work is still misleading concerning the difficulties of the proof of Theorem~\ref{th:1}.
Namely, a crucial role will be played in the proof by the results of part~I of this series (\cite{Pin2}) which is again heavily based on the results of another preparatory work (\cite{Pin}).
Finally we mention that apart from the relatively short final Section~9 all results of the present work (in many cases in a refined form) will be used in later parts of this series
to achieve further improvements over \eqref{eq:1.8}.

\section{Notation. The role of the explicit formula}
\label{sec:2}

The explicit formula proved in part I (\cite{Pin2}) will play a central role in the proof of Theorem~\ref{th:1}; in fact, it serves as the basis for any further examination.
In order to formulate the explicit formula we first need to introduce the notation.

Let $\ve$ and $\ve_0$ be small positive numbers, $X$ be a number large enough $(X > X_0(\ve, \ve_0))$, and let us define
\beq
X_1 := X^{1 - \ve_0}, \quad e(u) := e^{2\pi iu}, \quad S(\alpha) := \sum_{X_1 < p \leq X} \log pe(p\alpha),
\label{eq:2.1}
\eeq
where $p$, $p'$, $p_i$ will always denote primes.
$|\mathcal M|$ will denote the cardinality of the finite set $\mathcal M$.
We will define the major ($\mathfrak M$) and minor ($\mathfrak m$) arcs through the parameters $P$ and $Q$ satisfying
\beq
(\log X)^C \leq P \leq X^{4/9 - \ve}, \quad Q = \frac{X}{P},
\label{eq:2.2}
\eeq
\beq
\mathfrak M = \bigcup_{q \leq P} \bigcup_{\substack{a\\(a, q) = 1}} \left[ \frac{a}{q} - \frac1{qQ}, \frac{a}{q} + \frac1{qQ} \right], \quad \mathfrak m = \left[ \frac1{Q}, 1 + \frac1{Q}\right] \setminus \mathfrak M.
\label{eq:2.3}
\eeq

We will examine the number of Goldbach decompositions of even numbers $m \in [X / 2, X]$ in the form
\beq
R  (m) = \sum_{\substack{p + p' = m\\ p, p' \geq X_1}} \log p \cdot \log p' = R  _1(m) + R  _2(m),
\label{eq:2.4}
\eeq
where
\beq
R  _1(m) = \int\limits_{\mathfrak M} S^2(\alpha) e(-m \alpha) d\alpha, \qquad
R  _2(m) = \int\limits_{\mathfrak m} S^2(\alpha) e(-m\alpha) d\alpha.
\label{eq:2.5}
\eeq

The now standard treatment of the minor arcs (Parseval's theorem and the estimate of Vinogradov, reproved in a simpler way by Vaughan) gives
\beq
|R  _2(m)| \leq \frac{X}{\sqrt{\log X}} \quad \text{ for } \ P \leq X^{2/5}
\label{eq:2.6}
\eeq
apart from at most $C\frac{X}{P} \log^{10} X$ exceptional values~$m$ (see Section~5 of \cite{Pin2}, for example).

In order to formulate the explicit formula for the major arcs in Goldbach's problem we will define the set $\mathcal E = \mathcal E(H, P, T, X)$ of generalized exceptional singularities of the functions $L'/L$ for all primitive $L$-functions $\mod r$, $r \leq P$, as follows ($\chi_0 = \chi_0(\mod 1)$ is considered as a primitive character $\mod 1$)
\begin{align}
(\varrho_0, \chi_0) &\in \mathcal E \quad \text{ if } \ \varrho_0 = 1,
\label{eq:2.7}\\
(\varrho_i, \chi_i) &\in \mathcal E \quad \text{ if } \ \exists \chi_i,\ \mathrm{cond}\, \chi_i = r_i \leq P, \ L(\varrho_i, \chi_i) = 0,\nonumber\\
\beta_i &\geq 1 - \frac{H}{\log X}, \quad |\gamma_i| \leq T, \nonumber
\end{align}
where zeros of $L$-functions are denoted by $\varrho = \beta + i \gamma = 1 - \delta + i \gamma$ and $\cond \chi$ denotes the conductor of~$\chi$.
Let further
\begin{align}
A(\varrho) &= 1\phantom{-} \quad \text{ if } \ \varrho = 1,
\label{eq:2.8}\\
A(\varrho) &= -1 \quad \text{ if } \ \varrho \neq 1.
\nonumber
\end{align}

The expected main term of $R  _1(m)$ is the well-known singular series of Hardy and Littlewood, arising from the effect of the pole of $\zeta(s)$ at $s = 1$:
\beq
\mfS  (m) := \mfS (\chi_0, \chi_0, m) := \prod_{p\mid m} \left(1 + \frac1{p - 1}\right) \prod_{p \nmid m} \left( 1 - \frac1{(p - 1)^2}\right).
\label{eq:2.9}
\eeq
However, if we have zeros near to $s = 1$ then we necessarily have a number of secondary terms with coefficients $\mfS (\chi_i, \chi_j, m)$ corresponding to the primitive characters belonging to generalized exceptional zeros.
We will call these characters generalized exceptional characters, the corresponding singular series $\mfS (\chi_i, \chi_j, m)$ generalized exceptional singular series.
They can be expressed in a very complicated explicit form, proven in the Main Lemma of part~I (\cite{Pin2}).
However, the important properties of it can be incorporated into the following theorem, where we use the notation and conditions of the present section.

\medskip
\noindent
{\bf Theorem A (Explicit formula).} {\it Let $0 < \ve < \ve_0$, $\ve < \vartheta < \frac49 - \ve$ be any numbers, $2\mid m \in \left[ \frac{X}{2}, X\right]$.
Then there exists $P \in (X^{\vartheta - \ve}, X^\vartheta)$ such that for $X > X_0(\ve)$
\begin{align}
R  _1(m) &= \sum_{\varrho_i \in \mathcal E} \sum_{\varrho_j \in \mathcal E} A(\varrho_i) A(\varrho_j) \mfS (\chi_i, \chi_j, m) \frac{\Gamma(\varrho_i) \Gamma(\varrho_j)}{\Gamma(\varrho_i + \varrho_j)} m^{\varrho_i + \varrho_j - 1}
\label{eq:2.10}\\
&\quad + O \left( X e^{-cH} + \frac{X}{\sqrt{T}} + X^{1 - \ve}\right),
\nonumber
\end{align}
where the generalized singular series satisfy
\beq
|\mfS (\chi_i, \chi_j, m)| \leq \mfS (\chi_0, \chi_0, m) = \mfS (m);
\label{eq:2.11}
\eeq
further for any $\eta$ small enough
\beq
|\mfS (\chi_i, \chi_j, m)| \leq \eta ,
\label{eq:2.12}
\eeq
unless the following three conditions all hold,
\beq
r_i |C(\eta)m,\ r_j| C(\eta) m, \ \ \cond \chi_i \chi_j < \eta^{-3}
\label{eq:2.13}
\eeq
where $C(\eta)$ is a suitable constant depending only on~$\eta$.}

Its proof follows from Theorem~1 and Main Lemma~1 in part~I \cite{Pin2}.

\medskip
\noindent
{\bf Remark.}
A very important feature of the explicit formula is that the number $K$ of generalized exceptional zeros appearing in \eqref{eq:2.10} is by log-free zero density theorems (cf.\ Jutila \cite{Jut}) bounded from above by
\beq
K \leq Ce^{2H},
\label{eq:2.14}
\eeq
so it is bounded by an absolute constant (depending on $\ve$), if we choose $H$ as a sufficiently large absolute constant depending on $\ve$, which we suppose later on in the proof of Theorem~\ref{th:1}.
Similarly, we will choose $T$ as a sufficiently large constant depending on~$\ve$.

\medskip
Although the quoted explicit formula is in general a good starting point for the proof of
\beq
R  _1(m) > \ve \mfS (m) m
\label{eq:2.15}
\eeq
if $\vartheta$ is small enough, the argument breaks down in case of the existence of a Siegel-zero, $1 - \delta$ corresponding to $L(s, \chi_1)$, in which case we might have $\mfS (\chi_1, \chi_1, m) = -\mfS (m)$ and we cannot
show the crucial relation \eqref{eq:2.15} if $\delta$ is small enough.
In this case the Deuring--Heilbronn phenomenon can help.
This case was worked out as Theorem~2 in part I \cite{Pin2} which we quote now as

\medskip
\noindent
{\bf Theorem B.} {\it Let $\ve' > 0$ be arbitrary.
If $X > X(\ve')$, ineffective constant and there exists a Siegel zero $\beta_1$ of $L(s, \chi_1)$ with
\beq
\beta_1 > 1 - h/\log X, \ \ \cond \chi_1 \leq X^{\frac49 - \ve'},
\label{eq:2.16}
\eeq
where $h$ is a sufficiently small constant depending on $\ve'$, then}
\beq
E(X) < X^{\frac35 + \ve'}.
\label{eq:2.17}
\eeq

\medskip
\noindent
{\bf Remark.}
We will choose $\ve' = 10^{-3}$ here.
Then in the proof of Theorem~1 we are entitled to suppose that all $L(s, \chi)$ functions $\mod r \leq P$ satisfy
\beq
L(s, \chi) \neq 0 \quad \text{ for }\ s \in [1 - c_0 /\log X,\ 1]
\label{eq:2.18}
\eeq
if we choose $\vartheta \leq 0.44$.
 In other words, we do not need to worry about exceptional zeros $1 - \delta$ satisfying $\delta < c_0/\log X$ with a small but fixed~$c_0 > 0$.

\medskip
The well-known relation (cf.\ \cite{Kar}, p.~46) $(\mathrm{Re}\, w, \mathrm{Re}\, z > 0)$
\beq
\frac{\Gamma(w) \Gamma(z)}{\Gamma(w + z)} = B(w, z) = \int\limits^1_0 x^{w - 1} (1 - x)^{z - 1} dx
\label{eq:2.19}
\eeq
tells us that
\beq
|B(\varrho_i, \varrho_j)| \leq |B(\mathrm{Re}\, \varrho_i, \mathrm{Re}\, \varrho_j)| =
 B(1,1) + O(1 / \log X) = 1 + O(1/\log X).
\label{eq:2.20}
\eeq

Hence, taking into account the relations \eqref{eq:2.11}--\eqref{eq:2.13} we see that the estimation \eqref{eq:2.15} will follow, if we can show
\beq
\underset{\substack{\varrho_i, \varrho_j \in \mathcal E\\
(\varrho_i, \varrho_j) \neq (1,1)}}{\sum\nolimits^*} X^{-\delta_i - \delta_j} < 1 - 2\ve ,
\label{eq:2.21}
\eeq
where the $*$ means that the additional condition \eqref{eq:2.13} is satisfied for the pairs $(\varrho_i, \varrho_j)$ of zeros in the summation.

The expression \eqref{eq:2.21} can be estimated directly by density theorems and the Deuring--Heilbronn phenomenon, as done in the earlier estimates of Chen-Liu \cite{CL} and Hongze Li \cite{Li1}, \cite{Li2}, and Lu \cite{Lu}.
It also resembles the well-studied problem of the Linnik-constant, with the seemingly major disadvantage that
\[
\text{the zeros do not belong to a fixed modulus } q \leq P
\tag{\hbox{$\dagger$}}
\]
but to a set of different moduli $r_i \leq P$.

In what follows below, we will show that this disadvantage can be overwhelmed thanks to the information \eqref{eq:2.13} supplied by the explicit formula.

We will choose $P_0 = X^{\vartheta +2\ve}$, so our $P$ will satisfy
\beq
P \in \Bigl[ X^{\vartheta + \ve}, X^{\vartheta + 2\ve}\Bigr].
\label{eq:2.22}
\eeq
Thus the exceptional set arising from the minor arcs \eqref{eq:2.6} will be $o(X^{1 - \vartheta})$.
Then we consider the set $\mcR  $ of the $K$ generalized exceptional zeros appearing in \eqref{eq:2.10} whose number $K$ is bounded by an absolute constant depending on~$\ve$,
\beq
0 \leq K \leq K(\ve) - 1
\label{eq:2.23}
\eeq
according to \eqref{eq:2.14} since we will choose $H$ as a big constant depending on~$\ve$.
(If $K = 0$ we are ready.)

Let us choose now
\beq
\eta = \frac{\ve}{K^2(\ve)},
\label{eq:2.24}
\eeq
and write
\beq
C(\eta) = C_1(\ve).
\label{eq:2.25}
\eeq
In this case the total contribution of terms not satisfying \eqref{eq:2.13} will be really less than $\ve X$ in \eqref{eq:2.10}, so \eqref{eq:2.21} will really imply~\eqref{eq:2.15}.
Let us divide now the even numbers $m$ in $[X/2, X]$ into at most $2^{|\mcR  |}$ different classes $\mathcal M(\mcR')$ according to the subset $\mcR' \subset \mcR  $ of zeros which belong to primitive characters with moduli dividing~$C_1(\ve)m$
\beq
\mathcal M(\mcR') = \bigl\{ m \in [X/2, X],\ 2\mid m,\ r_i \mid C_1(\ve) m \Leftrightarrow r_i \in \mcR'\bigr\}.
\label{eq:2.26}
\eeq
(The subset might be empty for some $\mcR' \subset \mcR  $; for example, if $\varrho_i \in \mcR'$, $\varrho_j \notin \mcR'$, $r_i = r_j$, or if $\underset{r_i \in \mcR'}{\mathrm{l.c.m.}} [r_i] > X C_1(\ve)$.)

Let us denote
\beq
q(\mcR') := \mathrm{l.c.m.} [r_i; \ r_i \in \mcR'].
\label{eq:2.27}
\eeq

Now we can delete all classes $\mathcal M(\mcR')$ with
\beq
q(\mcR') > X^{\vartheta},
\label{eq:2.28}
\eeq
since in this case clearly
\beq
|\mathcal M(\mcR')| \leq C_1 (\ve) X^{1 - \vartheta},
\label{eq:2.29}
\eeq
and the number of all classes is
\beq
2^{|\mcR  |} \leq 2^{K(\ve)} = C_2(\ve).
\label{eq:2.30}
\eeq

Let us fix now any concrete class $\mcR'$ with
\beq
q := q(\mcR') \leq X^{\vartheta}.
\label{eq:2.31}
\eeq
Due to \eqref{eq:2.30} it is sufficient to restrict our attention now for values $m$ with
\beq
m \in \mathcal M(\mcR').
\label{eq:2.32}
\eeq
Hence, by \eqref{eq:2.21} it is sufficient to show for any $q \leq X^{\vartheta}$
\beq
S_0 = \underset{\substack{\varrho_i, \varrho_j \in \mathcal E\\
(\varrho_i, \varrho_j) \neq (0, 0)}}{\sum\nolimits^{**}} q^{-A(\delta_i + \delta_j)} < 1 - 2\ve
\label{eq:2.33}
\eeq
where $A = 1/\vartheta$ and the notation $**$ abbreviates now the condition
\beq
r_i \mid q, \ r_j \mid q, \ \ \cond(\chi_i \chi_j) < C_0(\ve) \ \left( = \frac{\ve^3}{K^6(\ve)}\right).
\label{eq:2.34}
\eeq
Thus we managed to get rid of the condition $(\dagger)$, and it is sufficient to consider characters modulo the same $q \leq X^{\vartheta}$.
Further advantages compared to the earlier treatments (\cite{CL}, \cite{Li1}, \cite{Li2}, \cite{Lu}) are that

(i) both zeros $\varrho_i$ and $\varrho_j$ run only through zeros with a bounded height $|\gamma| \leq T(\ve)$
and

(ii) the second zero $\varrho_j$ runs for every fixed~$\varrho_i$ only through zeros belonging to characters $\chi_j$ with
\beq
\cond(\chi_i \chi_j) < C_0(\ve).
\label{eq:2.35}
\eeq

Since zeros of $L(s, \chi)$ and $L(s, \overline\chi)$ are conjugate it will be simpler for us to change the condition~\eqref{eq:2.35} to consider further on the inequality
\beq
S_0 = \sum\nolimits' q^{-A(\delta_i + \delta_j)} < 1 - 2\ve
\label{eq:2.36}
\eeq
where the condition $\sum'$ will mean later on
\beq
(1,1) \neq (\varrho_i, \varrho_j) \in \mathcal E, \quad r_i \mid q, \ r_j\mid q, \ \ \cond(\chi_i \overline \chi_j) < C_0 (\ve).
\label{eq:2.37}
\eeq
This form makes the quasi-diagonal form of $S_0$ clear: only those pairs of zeros count where the relevant primitive characters are the same up to a character with a bounded conductor.

Let us use further on the notation (this will change the values of $H$ and $T$ by a factor $A$)
\beq
\log q = \mathcal L, \quad  \lambda_i = \delta_i \mathcal L \leq H,
\quad |\mu_i| = |\gamma_i \mathcal L| \leq T.
\label{eq:2.38}
\eeq
Then we can rewrite $S_0$ as
\beq
S_0 = \sum\nolimits^\prime e^{-A(\lambda_i + \lambda_j)} .
\label{eq:2.39}
\eeq

According to \eqref{eq:2.37} we will say that two generalized exceptional characters $\chi$ and $\chi'$ are equivalent, in notation $\chi \sim \chi'$ if there is a chain of generalized exceptional characters $\chi_\nu$ $(\nu = 1,2, \dots, n)$ such that $\chi_1 = \chi$, $\chi_n = \chi'$
\beq
\cond(\chi_\nu \overline{\chi_{\nu + 1}}) < C_0(\ve) \quad \text{ for }\ \nu = 1, \dots,  n - 1.
\label{eq:2.40}
\eeq
Such a chain has at most $K \leq K(\ve)$ characters in it; hence if $\chi$ and $\chi'$ are equivalent, then
\beq
\cond(\chi, \overline \chi') < C_3(\ve) = C_0(\ve)^{K(\ve) - 1}.
\label{eq:2.41}
\eeq

We remark here that since by Davenport \cite{Dav}, Ch.~14
\beq
\delta \gg \frac1{\sqrt{q} \log^2 q},
\label{eq:2.42}
\eeq
there is no generalized exceptional character $\chi \sim \chi_0(\mod 1)$, so the sum $S_0$ in \eqref{eq:2.39} in fact does not contain any pair of singularities $(1, \varrho)$, just pairs of zeros.
In such a way we can distribute the generalized exceptional zeros into $M$ $(\leq K)$ classes according to the equivalence classes $\mathcal H_\nu$ $(\nu = 1,2, \dots M)$ of the generalized exceptional characters.
Thus we obtain
\beq
S_0 \leq S := \sum^M_{\nu = 1} S^2_\nu,
\label{eq:2.43}
\eeq
where $S_\nu$ denotes the quantity
\beq
S_\nu := \sum_{\varrho_{\nu, j} \in \mathcal E,
\ \chi_{\nu, j} \in \mathcal H_\nu} e^{-A \lambda_{\nu, j}}.
\label{eq:2.44}
\eeq

According to this it will be important to introduce (and later estimate) the quantities
\beq
N(\Lambda) = \sum_{\varrho_i \in \mathcal E, \ \lambda \leq \Lambda} 1
\label{eq:2.45}
\eeq
and
\beq
\label{eq:2.46uj}
N_\nu(\Lambda) = \sum_{\chi \in {\mathcal H}_\nu, \ \lambda = \lambda_\chi \leq \Lambda} 1, \ \ \ \ \widetilde N(\Lambda) = \max_{\nu} N_\nu(\Lambda).
\eeq

\section{Methods}
\label{sec:3}

The reduction to zeros corresponding to characters modulo a fixed $q \leq P$,
the fact that it is sufficient to consider zeros with bounded height and the quasi-diagonal form \eqref{eq:2.37} of the critical sum \eqref{eq:2.33} are all new features compared with the earlier methods applied to the exceptional set in the previous works (cf. \cite{CL}, \cite{Li1}, \cite{Li2}, \cite{Lu}).
A further advantage is that \eqref{eq:2.33} shows now strong similarities to the case of the estimation of Linnik's constant; in fact, it looks like a ``two-dimensional'' variant of Linnik's problem.
This gives hope to apply the very powerful methods and/or results of Heath-Brown \cite{Hea} used by him to achieve the huge improvement $L \leq 5.5$ in the estimation of Linnik's constant compared to the earlier result $L = 13.5$ of Chen and Liu~\cite{CL2}.

The estimation of \eqref{eq:2.33} will be based on the following three principles, mentioned and used by Heath-Brown \cite{Hea}.

\noindent
{\bf Principle 1.} Zero-free region for $\prod\limits_{\chi(\mod q)} L(s, \chi)$.

\noindent
{\bf Principle 2.} Deuring--Heilbronn phenomenon.

\noindent
{\bf Principle 3.} `Log-free' zero density estimates.

\medskip
For the proof of the result $E(X) < X^{1 - \vartheta}$ it will suffice to take over from Heath-Brown's work \cite{Hea} a small part of his results concerning Principles~1 and~2 (see Theorems E, F, G in our next section) partially in the form improved by Xylouris \cite{Xyl}.
In the forthcoming papers, when proving sharper inequalities for $E(X)$ we will need much more results of this type and in many cases in somewhat stronger form.

On the other hand, the zero density estimates of \cite{Hea}, as well as some similar ones of others, used in earlier examination of Linnik's problem do not suite for our purposes.

Heath-Brown starts, namely, with a weighted average over primes, which does not seem to work in Goldbach's problem.
Since in this way in Linnik's problem zeros of the same $L$-function can be treated together, it is sufficient to estimate the number of $L$-functions $\mod q$, having at least one zero in a given range, instead of the total number of zeros in the relevant range as in case of usual density theorems.
The corresponding density theorems of Chen--Liu \cite{CL} and H. Z. Li \cite{Li1}, \cite{Li2} and Lu \cite{Lu} are far too weak as to yield Theorem~1.
Therefore we will show a new log-free density theorem (Theorem~C) which counts all zeros and is still just slightly weaker than the corresponding result of Heath-Brown \cite{Hea}, Lemma 11.1 which counts only the number of $\mathcal L$-functions belonging to these  zeros.

Another invention of Heath-Brown~\cite{Hea} is the proof of a `new density theorem', his Lemma~12.1, which works only for zeros very near to the line $\sigma = 1$ (approximately in the region $\sigma > 1 - 5/(4 \log q)$).
This result can also be extended for the number of $L$-functions having at least one zero in the relevant range.
The method of proof of this result is nearer to the proof of the Deuring--Heilbronn phenomenon than to that of the density theorems.
Concerning this result we succeeded in modifying the proof in such a way as to yield the same estimate without any loss for the total number of zeros (cf.\ Theorems H and I).
In another version of this method we can directly prove a weighted density theorem, essentially for the weights appearing in \eqref{eq:2.33} which is even more useful than unweighted density theorems (cf.\ Theorem J).

A new feature of our case is (which does not appear in Linnik's problem) that we need density theorems for zeros of a restricted class of $L$-functions belonging to equivalent characters (cf.\ \eqref{eq:2.40}, \eqref{eq:2.41} and \eqref{eq:2.46uj}).
The usual proofs for density theorems naturally work for these cases as well, and they usually yield somewhat stronger results in this case, like a comparison of Corollaries~\ref{cor:1} and \ref{cor:2} show in the next section.
However, an improvement of the technique applied by Heath-Brown in the proof of his new density theorem allows us to reach drastic improvements for the number of zeros in one equivalence class $(N_\nu(\Lambda))$ compared with the case of all $L$-functions $(N(\Lambda))$.
In the range $\sigma > 1 - 5/(4 \log q)$ $( \Leftrightarrow \lambda < 5/4)$, for example, we obtain at most $7$ zeros for one class, instead of the bound more than a hundred, supplied by the old or new density theorems of Heath-Brown (cf.\ \cite{Hea}, Tables 12 and 13) for the number of all $L$-function having at least one zero in the same range.
A further advantage of this method is that the bounds obtained in this way for $N_\nu(\Lambda)$ remain valid in the much wider range $\sigma > 1 - 6/\log q$.
After this, the contribution of zeros with $\lambda > 6$ can be estimated already very efficiently by Corollary~\ref{cor:2}.

\section{Auxiliary results}
\label{sec:4}
In the present section we will list the needed auxiliary theorems for the estimation of the crucial quantity for $S_0$ in \eqref{eq:2.33}.
These auxiliary theorems give important information about the distribution of zeros of $L$-functions near to $s = 1$.

The first one is a weighted density theorem which is the generalization of Heath-Brown's Lemma~11.1 \cite{Hea} for the case when we estimate the total number of zeros instead of the number of $L$-functions having at least one zero in the given region.
The result will have two different versions according to which we consider all $L$-functions or just a class of similar $L$-functions in the sense of Theorem~D.

\medskip
\noindent
{\bf Theorem C.} {\it Let $C', \ve, c_1, c_2, \Lambda > 0$ be given, $q \geq q(\ve, c_1, c_2, C', \kappa)$,
\beq
\Lambda_\infty = \frac13 \log\log \mathcal L, \ \ \varphi \geq \max_{\chi (\mod q)} \varphi(\chi), \ \
r = \frac{\varphi + c_1 + c_2}{3}, \ \ x_0 = 2\varphi + 3c_1 + c_2,
\label{eq:4.1}
\eeq
\beq
w(\varrho) = w\bigl(1 - \mathcal L(\lambda + i \mu) \bigr) = e^{-2(x_0 + \kappa) \lambda - \frac{r + \kappa}{2} d}, \quad
d = \max(0, \Lambda - \lambda).
\label{eq:4.2}
\eeq
Then we have with an absolute constant~$C$ depending on $C'$
\beq
\sum_{\substack{\varrho = \varrho_\chi, \ \chi (\mod q)\\
\lambda \leq \Lambda_\infty\\
|\mu| \leq C'}} w(\varrho) \leq (1 + C \ve) C_1 \sqrt{B_{\varphi, \kappa} (\Lambda) B_{\varphi, r}(\Lambda)}/\kappa = C_2 (\varphi, \kappa, \Lambda) (1 + C\ve),
\label{eq:4.3}
\eeq
where}
\beq
C_1 = \frac{2\varphi + 2c_1 + c_2}{2c_1 c_2},
\label{eq:4.4}
\eeq
\beq
B_{\varphi, \omega}(y) = \frac{\varphi}2 \left( \frac{1 - e^{-2\omega y}}{y} \right) + \left( \frac{1 - e^{-\omega y}}{y} \right)^2 .
\label{eq:4.5}
\eeq

\noindent
{\bf Theorem D.} {\it Suppose that $\mathcal K$ is a set of characters $\chi_i \mod q$ with the condition that for all pairs $\chi_i, \chi_j \in \mathcal K$
\beq
\cond(\chi_i \overline \chi_j) \leq q^\ve.
\label{eq:4.6}
\eeq
Further let us suppose the conditions of Theorem~C with
\beq
x_0 = \varphi + 3c_1 + c_2, \quad \varphi \geq \max_{\chi \in \mathcal K} \varphi(\chi).
\label{eq:4.7}
\eeq
Then \eqref{eq:4.3} holds if the sum is restricted for zeros of $L(s, \chi)$, $\chi \in \mathcal K$.}

\medskip
\noindent
{\bf Remark 1.} Although the estimate on the right-hand side of \eqref{eq:4.3} remained unchanged, the new estimate is stronger, since the new weights will be larger, due to the smaller choice of $x_0$ in \eqref{eq:4.7}.

Choosing $c_1 = 1/12$, $c_2 = 1/4$, $\varphi = 1/3$, $\kappa = 1/6$, ($C_1 = 26$) we are led to the following results (to be used in Section~9).

\begin{Cor}
\label{cor:1}
Let $\Lambda_0 = 1.311$, $\Lambda_1 = 2.421$, $\Lambda_2 = 3.96$, $\Lambda_3 = 5.8$,
$E_0 = 22.281$, $E_1 = 15.6$, $E_2 = 10.4$, $E_3 = 7.01$.
Then we have for $i = 0, 1,2$:
\beq
\sum_{\lambda \leq \Lambda_\infty} e^{-\frac83  \lambda - \frac{5}{12} \max(0, \Lambda - \lambda)} < E_i \quad \text{ for } \ \Lambda \geq \Lambda_i,
\label{eq:4.8}
\eeq
\beq
\sum_{\substack{\chi  \in \mathcal K\\
\lambda \leq \Lambda_\infty}} e^{-2 \lambda - \frac{5}{12} \max(0, \Lambda - \lambda)} < E_i \quad \text{ for } \ \Lambda \geq \Lambda_i.
\label{eq:4.9}
\eeq
\end{Cor}

\begin{Cor}
\label{cor:2}
With the notation of \eqref{eq:2.45}--\eqref{eq:2.46uj} and Corollary~\ref{cor:1} we have
\beq
N(\Lambda) < E_i e^{8 \Lambda / 3} \quad \text{ for } \ \Lambda \geq \Lambda_i,
\label{eq:4.10}
\eeq
\beq
N_\nu (\Lambda) < E_i e^{2 \Lambda} \quad \text{ for } \ \Lambda \geq \Lambda_i.
\label{eq:4.11}
\eeq
\end{Cor}

\noindent
{\bf Remark.} Since the functions $B_{\varphi, \omega}(y)$ are monotonically decreasing in $y$ for all non-negative values of the parameters $\varphi$ and $a$, the estimates $E_i$ arising from $\Lambda_i$ are valid for all $\Lambda \geq \Lambda_i$.

\medskip
The results listed as Theorems E, F, G below are Theorem~1 of Xylorius and Theorems 2 and 4 (more precisely Lemma~8.8) of Heath-Brown \cite{Hea} with the only change that the condition $|\gamma| \leq 1$ for the zeros can be substituted without any essential change in the proof for $|\gamma| \leq T$ if $q > q_0(T)$.
($T$ is in our case a large constant depending on~$\ve$.)
Thus in the following theorems let $\ve, T$ be positive constants,
\beq
M(s) = \prod_{\chi (\mod q)} L(s, \chi),
\label{eq:4.12}
\eeq
\beq
R(\alpha, T) = \left\{ s; \ \sigma \geq 1 - \frac{\alpha}{\log q}, \ |t| \leq T \right\}
\label{eq:4.13}
\eeq
and let us suppose that $q > q_0(T, \ve)$.

\medskip
\noindent
{\bf Theorem E.} {\it $M(s)$ has at most one zero in $R(0.44, T)$.
Such a zero, if it exists, is real and simple and corresponds to a non-principal character.}

\medskip
\noindent
{\bf Theorem F.} {\it $M(s)$ has at most two zeros, counted according to multiplicity, in $R(0.702, T)$.}

\smallskip
\noindent
{\bf Remark}. Heath-Brown \cite{Hea} proved this with $0.696$ in place of $0.702$.
The small improvement is due to Xylouris \cite[Tabellen 2, 3 and 7]{Xyl}.

\medskip
\noindent
{\bf Theorem G.} {\it Suppose that $\chi$ is a real non-principal character $\mod q$ with
\beq
L \left(1 - \frac{\lambda}{\log q}, \chi \right) = 0, \quad 0 < \lambda \leq 0.44.
\label{eq:4.14}
\eeq
Then $M(s)$ has only the zero $1 - \lambda / \log q$ in the region $R(\Lambda(\lambda), T) \cup R(1.18, T)$}
\beq
\Lambda(\lambda) = \min \left\{\left(\frac{12}{11} - \ve\right) \log \frac1{\lambda}, \frac13 \log\log\log q\right\}.
\label{eq:4.15}
\eeq

\smallskip
The fact that $M(s)$ has no other zeros in $R(\Lambda(\lambda), T)$ is exactly Lemma 8.8 of Heath-Brown~\cite{Hea}.
The absence of other zeros in $R(1.18, T)$ follows from Tables~4 and 7 of Heath-Brown \cite{Hea} (pp.\ 298, 301), which follow from his Lemmas 8.3 and 8.7, respectively.

In the following we will formulate the new density theorems which are improved forms of Lemma~12.1 of Heath-Brown~\cite{Hea}.
Although in the application we will work (unlike Heath-Brown) with a concrete function we will formulate the result more generally, similarly to \cite{Hea}.
(The following condition is the same as Conditions~1 and 2 of \cite{Hea} together.)

\medskip
\noindent
{\bf Condition 1.} Let $f$ be a non-negative continuous function from $[0, \infty)$ to $\mathbb R$, supported in $[0, t_0)$, twice differentiable on $(0, t_0)$ with $f''$ continuous and bounded.
Suppose that its Laplace transform
\beq
F(z) = \int\limits^{t_0}_0 e^{-zt} f(t) dt
\label{eq:4.16}
\eeq
satisfies
\beq
\mathrm{Re}\, F(z) \geq 0 \quad \text{ for } \ \text{\rm Re}\, z \geq 0.
\label{eq:4.17}
\eeq

\smallskip
We will work with the following pair of functions, satisfying Condition~1 (which appears in Lemma~7.2 of \cite{Hea})
\beq
g(u) = \frac1{30} (2 - u)^3 (4 + 6u + u^2) \quad \text{ with } \ u \in [0, 2]
\label{eq:4.18}
\eeq
\begin{align}
G(z) &= \int\limits^2_0 e^{-zu} g(u) du
\label{eq:4.19}\\
&= \frac{16}{15z} - \frac{8}{3z^3} + \frac{4}{z^4} - \frac{4}{z^6} + \frac{4e^{-2z}}{z^4} \left(\frac{z + 1}{z} \right)^2. \nonumber
\end{align}
The explicit form of $G(z)$ follows simply by computation. Further (cf.\ \cite{Hea}, Lemma~7.2)
\beq
g = h * h, \quad \text{ where } \ h(t) = 1 - t^2, \ \ -1 \leq t \leq 1.
\label{eq:4.20}
\eeq
Therefore we have
\beq
\mathrm{Re}(G(iy)) = \int\limits^2_0 g(u) \cos (uy) du = 2 \biggl( \int\limits^1_0 h(t) \cos ty\,dt \biggr)^2 \geq 0.
\label{eq:4.21}
\eeq
Finally, from Lemma~4.1 of \cite{Hea} (see also p.~279 of \cite{Hea}) we have by \eqref{eq:4.21}
\beq
\mathrm{Re}\, G(z) \geq 0 \quad \text{ for } \ \mathrm{Re}\, z  \geq 0.
\label{eq:4.22}
\eeq

Instead of the concrete functions $g(u)$ and $G(z)$ above we will use a one-parameter family of functions:
\beq
f = f_x(u) = xg(ux), \qquad u x \in [0, 2]
\Leftrightarrow u \in \left[0, \frac{2}{x}\right]
\label{eq:4.23}
\eeq
\begin{align}
F = F_x(z) &= \int\limits^{2/x}_0 e^{-zu} f_x(u) du = \int\limits^{2/x}_0 e^{-\frac{z}{x} ux} g(ux) x du
\label{eq:4.24}\\
&= \int\limits^2_0 e^{-\frac{z}{x} v} g(v) dv = G \Bigl(\frac{z}{x}\Bigr).\nonumber
\end{align}
An easy calculation shows (cf.\ Lemma~7.2 of \cite{Hea})
\beq
f(0) = \frac{16x}{15}, \quad F(0) = G(0) = \frac89, \quad F(-x) = G(-1) = \frac85.
\label{eq:4.25}
\eeq

In the applications (Theorems~I, H, J) the following additional property of the functions $F$ will have importance, which is satisfied for $G(z)$ for $A_0 = 13$, $B_0 = 1.25$, for example.

\medskip
\noindent
{\bf Condition 2.} There are non-negative constants $A_0$ and $B_0$ such that for any $t \in \mathbb R$
\beq
\mathrm{Re}\, \frac{G(a + it)}{G(a)} \geq \mathrm{Re}\, \frac{G(-b + it)}{G(-b)}
\label{eq:4.26}
\eeq
if
\beq
0 \leq a \leq A_0, \quad 0 \leq b \leq B_0.
\label{eq:4.27}
\eeq

\smallskip
\noindent
{\bf Remark.} Let us take $A_0$, $B_0$ arbitrary, non-negative constants, $\eta > 0$ fix.
Then for $\eta \leq a \leq A_0$, $0 \leq b \leq B_0$, $|t| > t_0 (\eta, A_0, B_0)$ we have from  \eqref{eq:4.19}
\beq
\mathrm{Re}\, \frac{G(a + it)}{G(a)} > \frac{c(A_0, \eta)}{t^2} > - \frac{c'(B_0)}{t^2} > \mathrm{Re}\, \frac{G(-b + it)}{G(-b)}
\label{eq:4.28}
\eeq
with positive constants $c(A_0, \eta)$, $c'(B_0)$.

This shows already that Condition~2 can be verified for the $G$ function in \eqref{eq:4.19} by the aid of computers for concrete values of $A_0$ and $B_0$ (consequently for the $F$-functions in \eqref{eq:4.24} if $x$, $A_0$ and $B_0$ are given).

This preparation makes possible to formulate the remaining 3 density theorems.
In these theorems we will use the notation \eqref{eq:2.45}--\eqref{eq:2.46uj}, where in \eqref{eq:2.46uj} we can ease the condition of equivalent characters without any loss for the final estimate.
Let $T > 0$, $\ve > 0$ be constants, $q > q(T, \ve)$, $\mathcal L = \log q$; zeros of $L$-functions belonging to characters $\chi$ $\mod q$ with $\varphi \geq \max \varphi(q)$ will be denoted (cf.\ \eqref{eq:2.38}) as
\beq
\varrho := \varrho_\chi := 1 - \delta + i\gamma, \quad \lambda:= \delta \mathcal L, \ \ \mu:= \gamma \mathcal L.
\label{eq:4.29}
\eeq
Suppose that with a $\lambda_0 \geq 0$ we have for all zeros of all $L(s, \chi)$, $\chi \, \mod q$
\beq
\lambda \geq \lambda_0.
\label{eq:4.30}
\eeq
We will use the following notation with $\lambda = \lambda_j \leq \Lambda$:
\beq
\frac{F(\lambda_j - \lambda_0)}{F(-\lambda_0)} = \psi_j, \quad
\frac{F(\Lambda - \lambda_0)}{F(-\lambda_0)} = \psi, \quad
\frac{f(0)\varphi}{2F(-\lambda_0)} = \xi, \quad \Delta = \psi - \xi > 0.
\label{eq:4.31}
\eeq

\medskip
\noindent
{\bf Theorem H.}
{\it Suppose that $f$ and $F$ satisfy Conditions~1 and~2.
Let $0 \leq \lambda_0 \leq B_0(F)$, $0 \leq \Lambda - \lambda_0 \leq A_0(F)$, $\Delta^2 > \xi + \ve$.
Then}
\beq
N(\Lambda) := N_T(\Lambda) = \sum_{|\gamma| \leq T, \ \lambda \leq \Lambda} 1 \leq \frac{1 - \xi}{\Delta^2 - \xi - \ve}.
\label{eq:4.32}
\eeq

\medskip
\noindent
{\bf Theorem I.} {\it Suppose that $\mathcal H$ is a set of characters $\chi\, \mod q$ with the property
\beq
\cond (\chi \overline\chi') \leq q^\ve \quad \text{ for } \ \chi, \chi' \in \mathcal H.
\label{eq:4.33}
\eeq
Then with the conditions of Theorem~H we have}
\beq
\widetilde N(\Lambda) = N_{T, \mathcal H} (\Lambda) = \sum_{\substack{
\chi \in \mathcal H\\
 |\gamma_\chi| \leq T, \ \lambda_\chi \leq \Lambda}} 1 \leq \frac{1 + \ve}{\Delta^2}.
\label{eq:4.34}
\eeq

\medskip
The following Theorem~J will enable us to estimate directly weighted sums over zeros, which arise in our problem.
This method partly reduces drastically the needed amount of calculations and also yields better estimates for the weighted sums than the usual treatment via partial summation (which cannot be performed easily in these cases due to the complicated forms of the upper bounds).
Further on we will suppose the Conditions of Theorem~H and additionally the existence of two constants $B$ and $C$ with
\beq
B > \max \bigl(C, t_0(f) \bigr) \geq 0
\label{eq:4.35}
\eeq
where $t_0 = t_0(f)$ is defined in Condition~A by $f(t) = 0$ for $t \geq t_0$.

First we state

\medskip
\noindent
{\bf Theorem J.} {\it Under the conditions of Theorem~H we have}
\beq
D(\Lambda) := \sum_{|\gamma_j| \leq T,\ \lambda_j \leq \Lambda}  (\psi_j - \psi) \leq \frac{1 - \xi + O(\ve)}{2\Delta}.
\label{eq:4.36}
\eeq

\medskip
Suppose that $F$ satisfies Condition~A, $d_0$ and $C'$ are given and with the $J$ unknown
variables $d_0 \geq d_1 \geq \dots \geq d_J \geq 0$
$(d_j = \Lambda - \lambda_j$, $j = 0, \dots, J)$ we know the upper bound
\beq
D^*(d_0) = \sum_{j \leq J} \bigl(F(d_0 - d_j) - F(d_0) \bigr) \leq C'.
\label{eq:4.37}
\eeq
We are interested in the maximal value of the quantity
\beq
S^* = \sum_{j \leq J} \bigl(e^{Bd_j} - e^{Cd_j}\bigr), \qquad B > C \geq 0,
\label{eq:4.38}
\eeq
under the additional constraint with given $\{e_j\}^J_1$
\beq
d_j \leq e_j, \quad e_1 \geq e_2 \geq \dots \geq e_J.
\label{eq:4.39}
\eeq
In this case we can find the maximum of~$S^*$ with a type of greedy algorithm as follows.
Suppose that $1 \leq r \leq J$ is defined as
\beq
\sum^{r - 1}_{j = 1} \bigl(F(d_0 - e_j) - F(d_0) \bigr) \leq C' < \sum^r_{j = 1} \bigl(F(d_0 - e_j) - F(d_0) \bigr).
\label{eq:4.40}
\eeq

\medskip
\noindent
{\bf Theorem K.} {\it Under the above conditions the optimal choice is ($r = 1$ is possible too, when \eqref{eq:4.41} is void) to choose
\beq
d_j = e_j \ \text{ for } \ j = 1,2, \dots, r - 1
\label{eq:4.41}
\eeq
and $d_r$ as the unique value with
\beq
F(d_0 - d_r) = F(d_0) + C' - \sum_{j = 1}^{r - 1} \bigl(F(d_0 - e_j) - F(d_0)\bigr).
\label{eq:4.42}
\eeq
Further, we choose $d_\nu = 0$ for $r < \nu \leq J$.}

\medskip
\noindent
{\bf Remark 1.}
If there is no $r \in (1, J)$ with \eqref{eq:4.40}, that is
\beq
\sum^J_{j = 1} \bigl(F(d_0 - e_j) - F(d_0) \bigr) \leq  C',
\label{eq:4.43}
\eeq
then the maximum of $S^*$ is clearly given by the choice $e_j = d_j$ for $1 \leq j \leq J$.

\medskip
\noindent
{\bf Remark 2.}
By the fact that $F$ is strictly monotonically decreasing if $f(t)$ is not identically~$0$, we obtain from \eqref{eq:4.42} that $d_r$ is completely determined and $0 \leq d_r \leq e_r$ due to $F(d_0 - e_r) \geq F(d_0 - d_r) \geq F(d_0)$.

\medskip
\noindent
{\bf Remark 3.}
Theorem~K itself refers to arbitrary numbers but in the applications we will use it with
\beq
d_j = \Lambda - \lambda_j, \ \ d_0 = \Lambda - \lambda_0, \ \ C' = F(-\lambda_0) \frac{1 - \xi + C \ve}{2\Delta}
\label{eq:4.44}
\eeq
in conjunction with estimate \eqref{eq:4.36} of Theorem~J.
Although Theorem~K itself does need only Condition~1 for $F$, Theorem~J needs both Conditions~1 and~2.

\medskip
\noindent
{\bf Remark 4.}
The conditions $d_i \leq e_i$, $i = 1, \dots J$ imply $J$ inequalities of the type $(m = 1,2, \dots, J)$
\beq
D^*(d_0, m) = \sum_{j \leq m} \bigl(F(d_0 - d_j) - F(d_0)\bigr) \leq C'_0(m)
\label{eq:4.45}
\eeq
with
\beq
C'_0(m) := C'_0(m; \{e_j\}) = \sum_{j \leq m} \bigl(F(d_0 - e_j) - F(d_0)\bigr).
\label{eq:4.46}
\eeq
Although it is not necessary for application in the present work (just in later parts of this series) we will examine the more general case when in place of the conditions $d_i \leq e_i$ we will have more generally a finite sequence of $M$ inequalities $(d_1 \geq \dots \geq d_M \geq 0)$
\beq
\widetilde D(d_0, m) = \sum_{j \leq m} F(d_0 - d_j) \leq c'(m) = c(m) + m F(d_0)
\label{eq:4.47}
\eeq
where $c(m) \geq c(m - 1)$ is an increasing sequence of real numbers satisfying for $m = 2, \dots M - 1$
\beq
c(m + 1) - c(m) \leq c(m) - c(m - 1) .
\label{eq:4.48}
\eeq
(If there exists an $m_0 = \min \{\nu; \ c(\nu) < 0\}$, then the system consists of at most $m_0 - 1$ elements.
In this case we cancel the inequalities with index $\geq m_0$, which is equivalent to the condition $M = m_0 - 1$, $c(j) \geq 0$ for $j = 1, \dots, M$.)

\medskip
\noindent
{\bf Theorem L.} {\it Under the above conditions the maximum of $S^*$ in \eqref{eq:4.38} (with $M = J$) is achieved for the uniquely determined sequence $\{d_m\}^M_{m = 1}$ for which equality holds in \eqref{eq:4.47} for all $m = 1, \dots, M$.}

\medskip
\noindent
{\bf Remark 5.}
If we add at the end some terms $d_m = 0$ with $m > M$ this does not change the value of~$S^*$, so allowing some extra conditions with $M < m \leq R$ ($R$ a fixed large constant) $c(m) = c(M)$, will definitely not decrease the maximum.
In this way we can work in the applications with an a priori determined (large but bounded) number $M$ of variables (since we know that in our application the number of zeros with $\lambda_j \leq \lambda$ satisfies a bound depending on $\lambda$ -- see e.g.\ Corollary~\ref{cor:2}) and the obtained upper bound for the new (extended) sum $S^*$ will constitute an upper bound for the original~$S^*$.

\medskip
\noindent
{\bf Remark 6.}
As we substituted the conditions $d_i \leq e_i$ by its consequence \eqref{eq:4.45}--\eqref{eq:4.46} this may theoretically increase the maximum value~$S^*$.
However, Theorem~L tells us that this is not the case since for the maximum configuration $\{d_i\}$ we will in fact have $d_i = e_i$ for $i \leq J$, since all the inequalities \eqref{eq:4.45} are sharp.

\section{Proof of Theorems C and D}
\label{sec:5}

In this section we will prove Theorems C and D.
The first part of the proof will be very similar to the proof presented in Section~3 of \cite{Pin} which follows the works of Heath-Brown \cite{Hea} and Graham \cite{Gra}.
The second part of the proof will use also ideas of Heath-Brown \cite{Hea}, however not from the proof of the corresponding density theorem, but from Section~13 of \cite{Hea}.

Following more closely \cite{Hea} and slightly different from (3.14) of \cite{Pin} we will use the notation
\beq
U \! =\! q^u, \ U_0 \! =\! q^{u_0}, \ V \!= \! q^v, \ W = q^w, \ X_1 = q^{x_1}, \ X = q^x, \ X_0 = q^{x_0},\ \overline U = U \mathcal L^2,
\label{eq:5.1}
\eeq
\beq
\gathered
\ve_1 = \ve^2 / 100,\ \ w = c_1 - \ve,\ \ v = u + c_2,\ \ x_0 = u_0 + r, \\
x_1 = x_0 + \kappa,\ \  u = u_0 + \eta,\ \ x = x_0 + \eta,\ \ 0 \leq \eta \leq \kappa.
\endgathered
\label{eq:5.2}
\eeq
Further, in case of Theorem C, we will choose
\beq
u_0 = \varphi + 2c_1, \quad x_0 = 2\varphi + 3c_1 + c_2, \quad r = \varphi + c_1 + c_2,
\label{eq:5.3}
\eeq
whereas in case of Theorem~D we set
\beq
u_0 = 2c_1, \quad x_0 = \varphi + 3c_1 + c_2, \quad r = \varphi + c_1 + c_2.
\label{eq:5.4}
\eeq

Similarly to \cite{Pin} and \cite{Hea} we will use Graham's weights
\begin{align}
\psi_d &= \begin{cases}
\mu(d) &\text{ for } \ 1 \leq d \leq \overline U\\
\mu(d) \frac{\log(V / d)}{\log (V / \overline U)} &\text{ for }\ \overline U \leq d \leq V\\
0 & \text{ for } \ d \geq V
\end{cases} ,
\label{eq:5.5} \\
\theta_d &= \begin{cases} \mu(d) \frac{\log W/d}{\log W} & \ \, \text{ for } \ 1 \leq d \leq W\\
0 & \ \, \text{ for } \ d \geq W \end{cases},
\label{eq:5.6}
\end{align}
\beq
\Psi(n) = \sum_{d\mid n} \psi_d, \qquad \vartheta (n) = \sum_{d \mid n} \theta_d,
\label{eq:5.7}
\eeq
further, the functions
\begin{align}
F(s) &= \sum_{i \leq V, \ j \leq W} \psi_i \theta_j \chi \bigl(([i,j])\bigr) [i,j]^{-s},
\ \ G_q(s) = \sum_{i \leq W, \ j \leq W} \theta_i \theta_j [i,j]^{-s},
\label{eq:5.8}\\
S(Y) &= \sum^\infty_{n = 1} \Psi(n) \vartheta(n) \chi(n) n^{-\varrho} e^{-n/Y} = \frac1{2\pi i} \int\limits_{(1)} L(s + \varrho, \chi) F(s + \varrho) \Gamma(s) Y^s ds,
\label{eq:5.9}
\end{align}
where $[i,j]$ denotes the least common multiple of $i$ and $j$ ($\mathrm{Re}\, \varrho = \beta$).
We move the line of integration in \eqref{eq:5.9} to $\mathrm{Re}\, s = 1 - \beta - 1/k$ with $k = \lceil 4 \ve^{-1}\rceil$, and obtain analogously to p.~318 of \cite{Hea} the estimate
\beq
S(X) \ll (q^\varphi V W X^{-1})^{1/k} q^{1/k^2} \mathcal L^3 X^{1 - \beta} \ll (q^\varphi VW X^{-1})^{1/k} q^{\frac2{k^2} - \ve_1} \ll q^{-\ve_1}
\label{eq:5.10}
\eeq
in view of $1 - \beta \leq \ve_1 / x$, $F(s + \varrho) \ll \sum\limits_{n \leq VW} d^2(n) n^{-1 + 1/k} \ll (VW)^{1/k} \mathcal L^3$ and
\beq
x - v - w - \varphi = x_0 - u_0 - c_2 - w - \varphi = \ve > 2/k,
\label{eq:5.11}
\eeq
analogously to (11.8)--(11.10) of \cite{Hea}, and (3.24)--(3.27) of \cite{Pin}.

The relation $\Psi(n) = 0$ for $2 \leq n \leq U \mathcal L^2 = \overline U$ implies
\beq
S(U) = e^{-1/U} + O \biggl(\sum_{n > \mathcal L^2 U} d^2(n) e^{-n/U} \biggr) = 1 + O\left(\frac1{U}\right).
\label{eq:5.12}
\eeq
This and \eqref{eq:5.10} yield
\beq
\sum^\infty_{n = 1} \Psi(n) \vartheta(n) \bigl(e^{-n/X} - e^{-n/U} \bigr) \chi(n) n^{-\varrho} = -1 + O (q^{-\ve_1}).
\label{eq:5.13}
\eeq
We will use now Hal\'asz' inequality in the simple form given by Lemma~1.7 of \cite{Mon}, with
\begin{align}
a_n &= \Psi(n) \vartheta(n) n^{-\frac12} \bigl(e^{-\frac{n}{X}} - e^{-\frac{n}{U}} \bigr)
\label{eq:5.14}\\
b_n &= \vartheta^2(n) \bigl(e^{-\frac{n}{X}} - e^{-\frac{n}{U}} \bigr), \quad s_j = \varrho_j - \frac12.
\nonumber
\end{align}
\[
\xi = \Bigl( \Psi(n)n^{-\frac12} \sqrt{e^{-n/X} - e^{-n/U}} \Bigr)^\infty_{n = 1} , \quad \varphi_j = \Bigl( \sqrt{b_n} \chi_j(n) n^{-s_j} \Bigr)^\infty_{n = 1}
\]
\[
(\varphi_j, \xi) = f(s_j, \chi_j) = \sum^\infty_{n = 1} a_n \chi_j(n) n^{-s_j}
\]
\[
(\varphi_i, \varphi_j) = B\bigl( s_i + \overline s_j, \chi_i \overline\chi_j\bigr), \quad
B(s, \chi) = \sum^\infty_{n = 1} b_n \chi(n) n^{-s}.
\]
Here we have analogously to \cite{Hea}, (11.14) and \cite{Pin}, (3.32)
\begin{align}
\|\xi\|^2 = \sum^\infty_{n = 1} \frac{|a_n|^2}{b_n} &= \bigl(1 + O (1 / \mathcal L) \bigr) \frac{2x - u - v}{2(v - u)}
= \left( 1 + O \left( \frac1{\mathcal L}\right)\right) \mathcal C_0, \label{eq:5.15}\\
\mathcal C_0 &= \frac{2\varphi + 2c_1 + c_2}{2c_2}.\nonumber
\end{align}
Any term with $\chi_i \overline\chi_j \neq \chi_0$ ($\chi_0$ is the principal character $\mathrm{mod}\, q$) will be, similarly to \eqref{eq:5.10}, in case of Theorem~C
\begin{align}
B(s_i + \overline s_j, \chi_i \overline\chi_j)
&= \sum^\infty_{n = 1} \vartheta^2(n) \chi_i \overline \chi_j(n) \bigl(e^{-\frac{n}{X}} - e^{-\frac{n}{U}} \bigr) n^{-(\varrho_i + \overline\varrho_j - 1)}
\label{eq:5.16}\\
&\ll (q^\varphi W^2 U^{-1})^{\frac1{k}} q^{\frac2{k^2}} \ll q^{-\ve_1} \nonumber
\end{align}
whereas in case of Theorem~D:
\beq
B(s_i + \overline s_j, \chi_i \overline \chi_j) \ll (W^2 U^{-1})^{\frac1k} q^{\frac{2}{k^3}} \ll q^{-\ve_1}.
\label{eq:5.17}
\eeq

If $\chi_i \overline \chi_j = \chi_0$, that is, $\chi_i = \chi_j$, then, similarly to (3.35)--(3.38) of \cite{Pin}, we have
\beq
B(s_i + \overline s_j, \chi_0) = \frac{\varphi(q)}{q} G_q(1) \Gamma(2 - \varrho_i - \overline \varrho_j) \bigl(X^{2 - \varrho_i - \overline \varrho_j} - U^{2 - \varrho_i - \overline \varrho_j}\bigr) + O (q^{-\ve_1})
\label{eq:5.18}
\eeq
where the real quantity $G_q(1)$ satisfies
\beq
\left| \frac{\varphi(q)}{q} G_q(1) \right| \leq \frac{1 + O(1 / \mathcal L)}{w \mathcal L},
\label{eq:5.19}
\eeq
by the Proposition after (3.36) in \cite{Pin}.

Until now we followed quite closely \cite{Hea}, Section 11 and \cite{Pin}.
The above considerations were valid for all values $\ve, c_1, c_2, \kappa, x_0, \eta$, which determine the values of the remaining parameters $u_0$, $x_0$, $v$, $w$, $u$ and $x$.
Now we will take an average over $\eta$ with $0 \leq \eta \leq \kappa$.

Using Hal\'asz's inequality (Lemma 1.7 of \cite{Mon}) with the notation
\beq
\gathered
w_j = w(\varrho_j) = e^{-2x_1\lambda_j - \frac{\kappa + r}{2} d_j} \\
 z_{i,j} = \mathcal L(2 - \varrho_i - \overline \varrho_j) = \mathcal L(\delta_i + \delta_j + i(\gamma_j - \gamma_i)) = \lambda_i + \lambda_j + i(\mu_j - \mu_i)
\endgathered
\label{eq:5.20}
\eeq
we obtain from the relations \eqref{eq:5.1}--\eqref{eq:5.4}, \eqref{eq:5.13}--\eqref{eq:5.19} with $C_1 = C_0/c_1$, after taking  average in $\eta \in [0,\kappa]$ which affects $x, u, X, U$.
\beq
(1 + O(\ve)) \biggl(\sum^J_{j = 1} w_j\biggr)^2 =
\frac{1 + O(\ve)}{\kappa} \int\limits_0^\kappa \biggl(\sum_{j = 1}^J w_j\biggr)^2 d\eta
\leq C_1 \sum^J_{j = 1} w_j \sum_{i \sim j} w_i |\mathcal H(z_{j, i})|
\label{eq:5.21}
\eeq
where we write $i \sim j$ if $\varrho_i$ and $\varrho_j$ are zeros of the same $L(s, \chi)$ and
\begin{align}
|\mathcal H(z)| :&= \frac{|\Gamma(z / \mathcal L)|}{\mathcal L} \cdot \frac1{\kappa}
\left| \bigl(e^{x_0 z} - e^{u_0 z}\bigr)  \frac{e^{\kappa z} - 1}{z} \right|
\label{eq:5.22}\\
&= \frac{1 + O(\ve)}{\kappa} \left| e^{(x_0 + \kappa) z} \frac{(1 - e^{-r z})(1 - e^{-\kappa z})}{z^2} \right| \nonumber\\
&\leq \frac{1 + O(\ve)}{\kappa} e^{x_1 a} \left| \frac{1 - e^{-rz}}{z}\right| \, \left|
\frac{1 - e^{-\kappa z}}{z} \right| =: \frac{1 + O(\ve)}{\kappa} \mathcal H_1(z) \nonumber
\end{align}
if $\mathrm{Re}\, z = a$.

We call the attention of the reader to the fact that while the value of the LHS of \eqref{eq:5.21} is independent of $\eta$ (up to $O(\ve)$) the RHS would actually depend on~$\eta$.
This dependence disappears only after taking the integral over $\eta$ and this phase is represented already in the form given on the RHS of \eqref{eq:5.21}.

In order to estimate for a given fixed zero $\varrho_j$ belonging to a given $\chi \neq \chi_0$, say, the sum over all terms $\mathcal H_1 (z_{j, i})$ $(\varrho_i = \varrho_\chi)$ we introduce the notation (cf.\ \cite{Hea}, p.~325) (with new parameters $\omega$ and $\lambda$, $\omega = \kappa$ or $r$)
\[
f_1(t) = \begin{cases} \sinh ((\omega - t)\lambda) &\text{for } 0 \leq t \leq \omega\\
0 & \text{for }  t \geq \omega
\end{cases},
\]
\begin{align}
F_{1, \omega}(z) = F_1(z) &= \int\limits^\infty_0 e^{-zt} f_1(t) dt = \frac12 \left\{ \frac{e^{\omega \lambda}}{\lambda + z} + \frac{e^{-\omega \lambda}}{\lambda - z} - \frac{2\lambda e^{-\omega z}}{\lambda^2 - z^2} \right\}
\label{eq:5.23}\\
F_{2, \omega}(z) = F_2(z) &= \left( \frac{1 - e^{-\omega z}}{z} \right)^2. \nonumber
\end{align}

As in (13.2) of \cite{Hea} we see that
\beq
\mathrm{Re}\, F_1(z) \geq \frac{\lambda e^{\omega\lambda}}{2} |F_2(\lambda + z)| \ \ \text{ for } \ \mathrm{Re}\, z \geq 0,
\label{eq:5.24}
\eeq
because the relation \eqref{eq:5.24} holds with equality for $\mathrm{Re}\, z = 0$ and therefore, by Lemma~4.1 of \cite{Hea}, \eqref{eq:5.24} holds for the whole halfplane~$\mathrm{Re}\, z \geq 0$.

Choosing the parameter $\lambda$ in \eqref{eq:5.23} as
\beq
\lambda = \lambda_j,
\label{eq:5.25}
\eeq
we obtain for a fixed $j$ by \eqref{eq:5.24} (the summation runs over zeros of $L(s, \chi)$)
\begin{align}
&\sum_{\substack{i; i \sim j\\
|\gamma_i - \gamma_j| < \mathcal L^{-1/2}}} \bigl| F_2(\lambda_j + \lambda_i + i(\mu_j - \mu_i) \bigr|
\label{eq:5.26}\\
&\leq \frac{2}{\lambda_j} e^{-\omega \lambda_j} \sum_{\substack{i; i \sim j\\ |\gamma_i - \gamma_j| < \mathcal L^{-1/2}}}  \mathrm{Re}\, F_1\bigl( (s_j - \varrho_i) \mathcal L\bigr) \nonumber
\end{align}
where $s_j = 1 + i\gamma_j$.

The contribution of all terms with $|\gamma_i - \gamma_\ell| > \mathcal L^{-1/2}$ (consequently $|\mu_i - \mu_j| > \mathcal L^{1/2}$)
 for any $\ell$ is clearly by well known log-free density theorems
\beq
\ll \Bigl(\sum w_j\Bigr)^2  e^{(x_0 + \kappa)} \frac{e^{3\Lambda_\infty}}{\mathcal L}
< \ve \Bigl( \sum w_i\Bigr)^2.
\label{eq:5.27}
\eeq
Applying the last displayed formula of the proof of Lemma 13.2 of \cite{Hea}:
\begin{align}
&\sum_{\substack{i; i \sim j\\ |\gamma_i - \gamma_j| < \mathcal L^{-1/2}}} \mathrm{Re}\, F_1\bigl((s_j - \varrho_i)\mathcal L\bigr)
\label{eq:5.28}\\
&\leq f_1(0) \left( \frac{\varphi}{2} + \ve \right) - \mathcal L^{-1} \sum^{\infty}_{n = 1} \Lambda(n) \mathrm{Re}\, \left(\frac{\chi(n)}{n^{s_j}} \right) f_1 \left( \frac{\log n}{\mathcal L} \right)
\nonumber\\
&\leq f_1(0) \left(\frac{\varphi}{2} + \ve \right) + \mathcal L^{-1} \sum^\infty_{n = 1} \Lambda(n) \frac{\chi_0(n)}{n} f_1 \left(\frac{\log n}{\mathcal L} \right) \nonumber\\
&\leq f_1(0) \left( \frac{\varphi}{2} + \ve\right) + F_1(0) + \ve.
\nonumber
\end{align}

So, we obtain finally from \eqref{eq:5.26}--\eqref{eq:5.28} for any fixed $j$
\begin{align}
\sum_{\substack{i; i \sim j\\ |\gamma_i - \gamma_j| < \mathcal L^{-1/2}}}
|F_2(z_{i,j})|
&\leq \frac{\varphi}{2} \left( \frac{1 - e^{-2\omega \lambda_j}}{\lambda_j} \right) + \left( \frac{1 - e^{-\omega \lambda_j}}{\lambda_j} \right)^2 + \ve
\label{eq:5.29}\\
&=: B_{\varphi, \omega} (\lambda_j) + \ve. \nonumber
\end{align}
This, together with \eqref{eq:5.21}, \eqref{eq:5.22}, \eqref{eq:5.23} and \eqref{eq:5.27} yields
\begin{align}
&(1 + O(\ve)) \biggl(\sum_j w_j\biggr)^2 \label{eq:5.30}\\
&= (1 + O(\ve)) \biggl( \sum_j e^{-2x_1 \lambda_j - \frac{r + \kappa}{2} d_j} \biggr)^2
\nonumber\\
& \leq C_1 \kappa^{-1} \sum_j \!\! \sum_{\substack{i \sim j\\
|\gamma_i - \gamma_j| < \mathcal L^{-1/2}}} \!\!
e^{-2x_1 (\lambda_i + \lambda_j) - \frac{r + \kappa}{2}(d_i + d_j)} e^{x_1
(\lambda_i \! + \! \lambda_j)} \sqrt{|F_{2,\kappa}(z_{i,j})F_{2,r}(z_{i,j})|}
\nonumber \\
& = C_1 \kappa^{-1} \sum_j \sum_{\substack{i \sim j\\ |\gamma_i - \gamma_j| < \mathcal L^{-1/2}}}
e^{-x_1 (\lambda_i + \lambda_j) - \frac{r + \kappa}{2}(d_i + d_j)}
\sqrt{|F_{2,\kappa}(z_{i,j})F_{2,r}(z_{i,j})|}.
\nonumber
\end{align}
Taking into account that
\beq
\frac{1 - e^{-x}}{x}
\label{eq:5.35}
\eeq
is monotonically decreasing for $x \geq 0$, we obtain that
\beq
e^{-\omega d_j} B_{\varphi, \omega} (\lambda_j) = B_{\varphi, \omega} (\lambda_j) \leq B_{\varphi, \omega} (\Lambda) \quad \text{ for } \ \lambda_j \geq \Lambda.
\label{eq:5.36}
\eeq
Further, as
\beq
\frac{e^x - e^{-x}}{x}
\label{eq:5.37}
\eeq
is monotonically increasing for $x \geq 0$, so is $e^{\omega\lambda}B_{\varphi,\omega}(\lambda)$.
Hence we obtain for $\lambda_j \leq \Lambda$
\beq
e^{-\omega d_j} B_{\varphi, \omega}(\lambda_j) = e^{-\omega \Lambda} e^{\omega \lambda_j} B_{\varphi, \omega} (\lambda_j) \leq B_{\varphi, \omega} (\Lambda).
\label{eq:5.38}
\eeq
Now using the trivial relation $|F_2(z_{i,j})| = |F_2(z_{j,i})|$, by $2ab \leq a^2 + b^2$ we obtain from \eqref{eq:5.35}--\eqref{eq:5.38}, \eqref{eq:5.29} and the Cauchy inequality
\begin{align}
&\sum_j \sum_{\substack{i \sim j\\
|\gamma_i - \gamma_j| < \mathcal L^{-1/2}}} e^{-x_1(\lambda_i + \lambda_j) - \frac{r + \kappa}{2} (d_i + d_j)}
\sqrt{|F_{2,\kappa}(z_{i,j})|\,|F_{2,r}(z_{i,j})|}
\label{eq:5.39}\\
&\leq \! \sum_j e^{-2x_1 \lambda_j - \frac{r + \kappa}{2} d_j}
\biggl(\!\sum_{\substack{i; i \sim j\\
|\gamma_i - \gamma_j| < \mathcal L^{-1/2}}}\!\!\!\! e^{-\kappa d_j} |F_{2,\kappa}(z_{i,j})|\!\biggr)^{\!\!\frac12}
\!\biggl(\!\sum_{\substack{i; i\sim j\\ |\gamma_i - \gamma_j| < \mathcal L^{-1/2}}}\!\!\!\! e^{-rd_j}F_{2,r}(z_{i,j})\!\!
\biggr)^{\!\! \frac12}
 \nonumber\\
&\leq \biggl(\sum_j w_j\biggr) \Bigl( \sqrt{B_{\varphi, \kappa}(\Lambda)B_{\varphi, r} (\Lambda)} + \ve\Bigr).
\nonumber
\end{align}

Consequently, from \eqref{eq:5.30} and \eqref{eq:5.39} we have
\beq
\sum_j w_j \leq (1 + O(\ve)) C_1 \sqrt{B_{\varphi, \kappa}(\Lambda) B_{\varphi, r}(\Lambda)}/\kappa =
(1 + O(\ve))C_2(\varphi, \kappa, \Lambda),
\label{eq:5.40}
\eeq
which proves Theorems C and D.

\section{Properties of the $G$-function}
\label{sec:6}
The following two lemmas show that the problem of showing Condition~2 for the $G$-function defined in \eqref{eq:4.19} can be reduced to its validity in a bounded region.
In the first 3 lemmas we will use explicit forms of $G(z)$ and $G'(z)$ as follows:
\begin{align}
G(z) &= \frac{16}{15z} - \frac{8}{3z^3} + \frac{4}{z^4} - \frac{4}{z^6} + \frac{4e^{-2z}}{z^4} \left(\frac{z + 1}{z}\right)^2,
\label{eq:6.1}
\\
G'(z) &= - \frac{16}{15z^2} + \frac8{z^4} - \frac{16}{z^5} + \frac{24}{z^7} - \frac{8e^{-2z}}{z^4} \left( 1 + \frac4{z} + \frac6{z^2} + \frac3{z^3} \right).
\label{eq:6.2}
\end{align}

The integral form of $G(z)$ in \eqref{eq:4.19} and $g(u) \geq 0$ trivially implies that for real $x \in R$
\beq
G(x) > 0, \quad G'(x) < 0, \quad G''(x) > 0, \dots\ .
\label{eq:6.3}
\eeq

Let $z = a + it$ and let us examine for fixed $t$ the behaviour of the functions
\begin{align}
\Psi_t(a) &= \frac{\mathrm{Re}\, G(a + it)}{G(a)}
\label{eq:6.4}
\\
\Phi_t(a) &= \mathrm{Re}\, G'(z) \cdot G(a) - \mathrm{Re}\, G(z) \cdot G'(a) = G^2(a) \frac{d\Psi_t(a)}{da}.
\label{eq:6.5}
\end{align}

\begin{Lem}
\label{lem:2}
For $0 \leq a \leq 13$, $|t| \geq 14$ we have
\beq
\Phi_t(a) > \frac{9G(a)}{|z|^4} > 0.
\label{eq:6.6}
\eeq
\end{Lem}

\noindent
{\bf Proof.} Since for $\mathrm{Re}\, z \geq 0$, we have $\mathrm{Re}\, G(z) \geq 0$ (cf.\ \eqref{eq:4.22}) and $0 < G(a) < 1$ it is sufficient to prove
\beq
\mathrm{Re}\, G'(z) > \frac{9}{|z|^4}.
\label{eq:6.7}
\eeq
From \eqref{eq:6.2} and \eqref{eq:4.25} we obtain by $|e^{-2z}| \leq 1$, $G(a) \leq G(0) = 8/9$ as claimed above.
Further,
\begin{align}
\mathrm{Re}\, G'(z) & \geq \frac{16}{|z|^4} \left( \frac1{15} (t^2 - a^2) - 1 - \frac3{|z|} - \frac3{|z|^2} - \frac3{|z|^3} \right)
\label{eq:6.8}\\
&\geq \frac{16}{|z|^4} \left( \frac{27}{15} - 1 - \frac3{14} - \frac3{14^2} - \frac3{14^3} \right) > \frac9{|z|^4}
\nonumber
\end{align}

\begin{Lem}
\label{lem:3}
For $a = -b \in [-1.25, \, 0]$, $|t| \geq 50$ we have
\beq
\Phi_t(a) > \frac{G(-b)}{25|z|^2}.
\label{eq:6.9}
\eeq
\end{Lem}

%%ITT TARTOK 20OLD.

\noindent
{\bf Proof.} We will use the notation $h(b) = \frac{8}{9} G(-b) + bG'(-b)$. Then
\beq
h'(b) = \frac{G'(-b)}{9} -b G''(-b) < 0,
\label{eq:6.10}
\eeq
and hence
\beq
h(b) \geq h(1.25) > 1/90 \ \text{ for } \ b \in [0,\, 1.25].
\label{eq:6.11}
\eeq
Further, by simple computation
\beq
G(-b) \geq G(0) = \frac89, \ |G'(-b)| \leq |G'(-1.25)| < 1.36 .
\label{eq:6.12}
\eeq
Now, from \eqref{eq:6.1}, \eqref{eq:6.2}, \eqref{eq:6.5}, and \eqref{eq:6.10}--\eqref{eq:6.12}
\beq
G(-b) \geq \max \left( \frac98 b|G'(-b)|, \, 0.65|G'(-b)|\right)
\label{eq:6.13}
\eeq
we obtain
\begin{align}
&\Phi_t(-b)\label{eq:6.14}\\
&\geq
\frac{16 G(-b)}{15 |z|^2} \left( \frac{1\! -\! (b / t)^2}{1\! +\! (b / t)^2} - \frac{7.5}{|z|^2}
\left(\! 1 + \frac2{|z|} + \frac3{|z|^3}\! + \! e^{2.5}\! \left(\! 1\! + \frac4{|z|} + \frac6{|z|^2} + \frac3{|z|^3} \! \right) \!\right) \!\right) \nonumber\\
&\quad - \frac{16|G'(-b)|}{15 |z|^2} \left( b +  \frac{3.75}{|z|^2} \left(1 + \frac{2}{|z|^2}+ \frac{2b}{3|z|^2} + e^{2.5} \frac{t^2 + 1}{t^2}
   \right)\right) \nonumber\\
& \geq \frac{16 G(-b)}{15|z|^2} \left(0.957 - \frac89 - 1.55 \cdot 0.015\right) > \frac{G(-b)}{25|z|^2}.
 \nonumber
\end{align}

The following lemma reduces the range for numerical check

\begin{Lem}
\label{lem:3masodik}
If $|t| \geq 8$, $\text{\rm Re }z = -b \in [-1.25, -0.14]$, then
\beq
\mathrm{Re}\, G(z) < - \frac1{140|z|^2} .
\label{eq:6.15}
\eeq
\end{Lem}

\noindent
{\bf Proof.} Since $\left|\frac{z + 1}{z} \right|^2 = \frac{(1 - b)^2 + t^2}{6^2 + t^2} \leq 1 + \frac1{|z|^2} \leq \frac{65}{64}$ and $\mathrm{Re}\, z^{-3} = |z|^{-6} b(3t^2 - b^2) > 0$ we obtain
\begin{align}
\mathrm{Re}\, G(z) &\leq \frac1{|z|^2} \left( - \frac{16}{15} b + \frac4{|z|^2} \left( 1 + \frac1{|z|^2} + e^{2b} \left| \frac{z + 1}{z} \right|^2 \right) \right)
\label{eq:6.16}\\
&\leq \frac{-1}{15|z|^2} \left(16 b - \frac{60}{64} \cdot \frac{65}{64} (e^{2b} + 1)\right).\nonumber
\end{align}

The function
\beq
h(b) = 16b - \frac{975}{1024} (e^{2b} + 1)
\label{eq:6.17}
\eeq
is increasing for $b < b_0 = \frac12 \ln \frac{8192}{975}$ and decreasing for $b > b_0$, so its minimum in $[0.14, 1.25]$ is
\beq
\min\bigl(h(0.14), h(1.25)\bigr) = 0.028\dots > 0.
\label{eq:6.18}
\eeq
\hfill Q.E.D.

\goodbreak
Now we will prove a lemma, which proves Condition~2 for the restricted range $|t| \leq \pi/2$ for $A_0, B_0 = \infty$ and for an arbitrary function $F(z)$ satisfying
\beq
F(z) = \int\limits^2_0 e^{-zv} f(v) dv \quad \text{ with }\ f(v) \geq 0.
\label{eq:6.19}
\eeq

\begin{Lem}
\label{lem:5}
$\Phi_t(x) = \mathrm{Re}\, F(x + it) / F(x)$ is monotonically increasing in $x$ for all $x \in \mathbb R$,
if $|t| \leq \pi/2$ and $F$ satisfies \eqref{eq:6.19}.
\end{Lem}

\noindent
{\bf Proof.} Let
\beq
k(v) = f(v) e^{-xv}, \quad
h(u, x) = \int\limits^u_0 k(v) dv, \quad q_u(x) = \frac{h(u, x)}{F(x)}.
\label{eq:6.20}
\eeq
From \eqref{eq:6.19} we obtain by partial integration
\begin{align}
\mathrm{Re}\, F(x + it) &= \bigl[h(u,x) \cos(ut)\bigr]^2_0 - \int\limits^2_0 h(u,x)(- t \sin (ut)) du
\label{eq:6.21}\\
&= F(x) \cos 2t + \int\limits^2_0 h(u,x) t \sin (ut) du,\nonumber
\end{align}
\beq
\Phi_t(x) = \cos 2t + \int\limits^2_0 q_u(x) (t \sin (ut) ) du.
\label{eq:6.22}
\eeq
Since for $|t| \leq \pi/2$, $u \in [0,2]$ we have $t \sin (ut) \geq 0$.
Hence, in order to show the lemma, it is sufficient to prove
\beq
\frac{d}{dx} q_u(x) \geq 0 \quad \text{for } \ u \in [0,2].
\label{eq:6.23}
\eeq
The property $f(v) \geq 0$ implies
\begin{align}
F^2(x) \frac{d}{dx} q_u(x) &= -\int\limits^u_0 k(v) v d v\int\limits^2_0 k(y) dy + \int\limits^u_0 k(y) dy \int\limits^2_0 k(v) v dv
\label{eq:6.24}\\
&= \int\limits^2_u k(v) v dv \int\limits^u_0 k(y) dy - \int\limits^2_u k(y) dy \int\limits^u_0 k(v) v dv \nonumber\\
&\geq u \biggl( \int\limits^2_u k(v) dv \int\limits^u_0 k(y) dy - \int\limits^2_u k(y) dy \int\limits^u_0 k(v) dv\! \biggr)\! = 0. \nonumber
\end{align}
\hfill Q.E.D.

Finally, we can check the remaining range by computer and verify Condition~2 for the $G$-function with $A_0 = 13$, $B_0 = 1.25$.

\begin{Lem}
\label{lem:6}
The function $G(z)$ in \eqref{eq:4.19} satisfies
\beq
\frac{\mathrm{Re}\, G(a + it)}{G(a)} \geq \frac{\mathrm{Re}\, G( - b + it)}{G(-b)}
\label{eq:6.25}
\eeq
for any $t \in R$ if\/ $0 \leq a \leq 13$, $0 \leq b \leq 1.25$.
\end{Lem}

The proof follows from Lemmas~\ref{lem:2}--\ref{lem:5} and from a computer check of \eqref{eq:6.25} (using Maple) for
\begin{alignat}{3}
a &= 0, \ \ & &0 \leq b \leq 0.14 \ \ & &\text{ for } \ \,14 \leq |t| < 50,
\label{eq:6.26}\\
0 &\leq a \leq 13, \ \ & &0 \leq b \leq 0.14 \ \ & &\text{ for  } \,
\phantom{88}8 \leq |t| < 14,
\label{eq:6.27}\\
0 &\leq a \leq 13, & &0 \leq b \leq 1.25 & &\text{ for } \pi/2 < |t| < 8.
\label{eq:6.28}
\end{alignat}

%%%IDE??

%%numberwithin{equation}{section}\setcounter{section}{6}
\section{Proof of Theorems H and I}
\label{sec:7}

Let
\beq
\varrho_j = \beta_j + i \gamma_j = 1 - \delta_j + i\gamma_j = 1 + \mathcal L^{-1} (-\lambda_j + i \mu_j), \quad j = 1,2,\dots, N = N(\lambda)
\label{eq:7.1}
\eeq
be the zeros of the $  L(s, \chi_j)$ functions $\mathrm{mod}\, q$ (counted with multiplicity if
$  L(\varrho, \chi) =   L(\varrho, \chi')$ or $\varrho$ is a multiple zero of some
$  L(s, \chi)$)
\beq
\lambda_0 \leq \lambda_j \leq \Lambda, \quad |\gamma_j| \leq L \Leftrightarrow |\mu_j| \leq \mathcal L L,
\quad L \leq \mathcal L.
\label{eq:7.2}
\eeq
Since $\zeta(s)$ has no zero in the region above we can assume $\chi \neq \chi_0$.
The notation $k \sim j$ will denote that $\varrho_k$ and $\varrho_j$ are zeros of the same $  L(s, \chi)$.
Further, suppose that the $  L$-functions belonging to the distinct characters $\chi^{(\nu)}$ $(1 \leq \nu \leq m)$ have exactly $N_\nu$ zeros in the above region (counted with multiplicity).
Then clearly $N = N_1 + N_2 + \dots + N_m$.
Let us denote the set of zeros of $  L(s, \chi)$ in \eqref{eq:7.2} by $Z(\chi)$.
Further, let for any $j \in [1, N]$ and with a function $F = F_x = G(z/x)$ (cf.\ \eqref{eq:4.16}--\eqref{eq:4.17} and
\eqref{eq:4.26}--\eqref{eq:4.27})
\beq
a_{k,j} = \frac{\mathrm{Re}\, F(\lambda_k - \lambda_0 + i (\mu_j - \mu_k))}{F(\lambda_k - \lambda_0)}, \quad
b_{k,j} = \frac{\mathrm{Re}\, F(-\lambda_0 + i (\mu_j - \mu_k))}{F(-\lambda_0)}
\label{eq:7.3}
\eeq
\beq
A_k = \underset{\substack{j\\ j \sim k}}{\sum\nolimits'} a_{k, j} , \qquad B_k = \underset{\substack{j\\ j \sim k}}{\sum\nolimits'} b_{k, j} , \qquad \psi_k = \frac{F(\lambda_k - \lambda_0)}{F(-\lambda_0)} ,
\label{eq:7.4}
\eeq
\beq
\gathered
N'_\ell = \sum_{\varrho_k \in Z(\chi_\ell)} A_k, \quad N' = \sum^m_{\ell = 1} N'_\ell = \sum^N_{j = 1} A_j,
\\
\psi = \frac{F(\lambda - \lambda_0)}{F(-\lambda_0)}, \quad \xi
= \frac{f(0)\varphi}{2F(-\lambda_0)} , \ \ \Delta = \psi - \xi,
\endgathered
\label{eq:7.5}
\eeq
where the $\sum'$ sign means in \eqref{eq:7.4} the extra condition $|\lambda_j + i(\mu_j - \mu_k)| < \mathcal L\delta$
where $\delta = \delta(\ve)$ is a sufficiently small constant.

Conditions 1, 2 (see \eqref{eq:4.16}--\eqref{eq:4.17} and \eqref{eq:4.26}--\eqref{eq:4.27}) and the definitions show that
\beq
a_{k,j} \geq b_{k,j}, \quad a_{k,j} \geq 0, \quad a_{k,k} = 1,
\label{eq:7.6}
\eeq
consequently for every $k = 1,2,\dots, N$
\beq
A_k \geq B_k, \quad A_k \geq 1, \quad N'_\ell \geq N_\ell, \quad N' \geq N.
\label{eq:7.7}
\eeq

Let $K(s, \chi)$ be defined as in \cite{Hea} (p. 285, after (6.2))
\beq
\label{eq:7.8uj}
K(s, \chi) = \sum_{r = 1}^\infty \Lambda(n) \text{\rm Re } \left(\frac{\chi(n)}{n^s}\right) f \left(\frac{\log n}{\mathcal L}\right)
\eeq
with a function $f(u) = f_x(u) = xg(ux)$ as in \eqref{eq:4.18} and \eqref{eq:4.23} connected to $F(z)$ by \eqref{eq:4.24}.
(We will omit the lower index $x$ to $f$, $F$ and $K$ which might change often depending on the particular problem.)

Following \cite{Hea}, Section 12 we will apply Lemma~5.2 of \cite{Hea} with the above function $K$.
So we obtain for any $\varrho_j$ in \eqref{eq:7.1} with $\beta_0 = 1 - \mathcal L^{-1} \lambda_0$ by $\lambda_j \geq \lambda_0$
\beq
K(\beta_0 + i \gamma_j, \chi_j) \leq - \mathcal L\!\! \!\! \sum_{\substack{k\\ k \sim j\\
|\lambda_k + i (\mu_j - \mu_k)| < \mathcal L\delta}}\!\!\!\!\! \mathrm{Re}\, F\bigl(\lambda_k - \lambda_0
+ i (\mu_j - \mu_k) \bigr) + f(0) \left(\frac{\varphi}{2} + \frac{\ve}{2} \right) \mathcal L.
\label{eq:7.8}
\eeq
Extending the summation for all zeros in \eqref{eq:7.2} with $|\lambda_k + i(\mu_j - \mu_k)| \geq \mathcal L\delta$
and using the relations $0 \leq \lambda_k - \lambda_0 \leq 13x$ we have in these cases by \eqref{eq:6.1}
\beq
F\bigl(\lambda_k - \lambda_0 + i(\mu_j - \mu_k) \bigr) \ll \frac1{|\mu_j - \mu_k|} \ll
\frac1{\mathcal L \delta}.
\label{eq:7.9}
\eeq
Since the number of terms being uniformly bounded by Jutila's density theorem
\beq
N(1 - \Lambda / \mathcal L, \mathcal L, q) \ll_\ve (q\mathcal L)^{(2 + \ve)
\Lambda / \mathcal L} \ll e^{3\Lambda} \ll 1,
\label{eq:7.10}
\eeq
including the other zeros in the summation on the right-hand side of \eqref{eq:7.8}
leads to an additional error of size $O(\delta^{-1}) = o(\mathcal L)$.
Thus we obtain from this modified form of \eqref{eq:7.8}, after summation for all~$ j \in (1,N)$
\begin{align}
&\mathcal L \Biggl\{ \sum_{j \leq N} \Biggl(\biggl( \sum_{\substack{k\\
k \sim j}} a_{k,j} \psi_k\biggr) F(-\lambda_0) - \frac{f(0) \varphi}{2} - \ve \Biggr) \Biggr\}
\label{eq:7.11}\\
&\leq - \sum_{j \leq N} K(\beta_0 + i\gamma_j, \chi_j) \nonumber\\
& = -\sum^\infty_{n = 1} \Lambda(n) \chi_0(n) n^{-\beta_0} f(\mathcal L^{-1} \log n)
\mathrm{Re} \biggl\{ \sum_{j \leq N} \chi_j(n) n^{-i\gamma_j}\biggr\}\nonumber\\
&\leq \sum^\infty_{n = 1} \Lambda(n) \chi_0(n) n^{-\beta_0} f(\mathcal L^{-1} \log n)
\biggl|\sum_{j \leq N} \chi_j(n) n^{-i\gamma_j} \biggr|.
\nonumber
\end{align}
Using \eqref{eq:4.31} we obtain
\beq
F( - \lambda_0) \geq F(\Lambda - \lambda_0) \geq f(0)(\varphi/2 + \ve).
\label{eq:7.12}
\eeq
Interchanging the order of summation on the left-hand side of \eqref{eq:7.11},
we obtain from \eqref{eq:7.11}, by $A_k \geq 1$
\beq
\mathcal L^2 \Biggl\{ \sum_{k \leq N} \biggl( \psi_k \biggl(\sum_{\substack{j \\
j \sim k}} a_{k,j} \biggr) F(-\lambda_0) - f(0) \varphi/2 - \ve\biggr)
\Biggr\}^2 \leq \sum\nolimits_1 \sum\nolimits_2,
\label{eq:7.13}
\eeq
where, using Lemma 5.3 of \cite{Hea}
\beq
\sum\nolimits_1 = \sum^\infty_{n = 1} \Lambda(n) \chi_0(n) n^{-\beta_0} f(\mathcal L^{-1} \log n)
= K(\beta_0, \chi_0) = \mathcal L\bigl(F(-\lambda_0) + o(1)\bigr)
\label{eq:7.14}
\eeq
and
\begin{align}
\sum\nolimits_2 &= \sum^\infty_{n = 1} \Lambda(n) \chi_0(n) n^{-\beta_0} f(\mathcal L^{-1} \log n)
\biggl| \sum_{j \leq N} \chi_j(n) n^{-i\gamma_j} \biggr|^2
\label{eq:7.15}\\
&= \sum_{j,k \leq N} K\bigl(\beta_0 + i(\gamma_j - \gamma_k), \chi_j \overline \chi_k\bigr)
\nonumber
\end{align}
since the above value is real.
By Lemma~5.3 of \cite{Hea} we have by \eqref{eq:7.3}--\eqref{eq:7.7} for any fixed $k$ for the terms with $j \sim k$ a sum
\begin{align}
\sum_{\substack{j\\
j \sim k}} K \bigl(\beta_0 + i(\gamma_j \! - \! \gamma_k), \chi_0\bigr)
&= \sum_{\substack{j\\ j \sim k}} \mathcal L \bigl\{\mathrm{Re}\, F(-\lambda_0 + i(\mu_j - \mu_k)) + o(1)\bigr\}
\label{eq:7.16}\\
&= \mathcal L \bigl\{ B_k F(-\lambda_0) + o(1)\bigr\} \leq \mathcal L\bigl\{A_k F(-\lambda_0) + o(1) \bigr\}.
\nonumber
\end{align}
Again, by Lemma~5.2 of \cite{Hea}, we obtain in case of Theorem~H for the total contribution
of all other terms the estimate
\beq
\sum_{\substack{j, k \leq N\\
k \not\sim j}} K\bigl(\beta_0 + i(\gamma_j - \gamma_k), \chi_j \overline \chi_k\bigr)
\leq \mathcal L \left(\frac{f(0)\varphi}{2} + \ve\right) \sum_{\substack{\kappa, \nu \leq m\\
\kappa \neq \nu}} N_\nu N_\kappa,
\label{eq:7.17}
\eeq
while in case of Theorem~I we obtain
\beq
\sum_{\substack{j, k \leq N\\
k \not\sim j}} K\bigl(\beta_0 + i(\gamma_j - \gamma_k), \chi_j \overline\chi_k\bigr)
\leq \ve \mathcal L \sum_{\substack{\kappa, \nu \leq m\\
\kappa \neq \nu}} N_\nu N_\kappa.
\label{eq:7.18}
\eeq

Dividing \eqref{eq:7.13} by $(\mathcal L F(-\lambda_0))^2$ we obtain from \eqref{eq:7.14}--\eqref{eq:7.17}
with the choice of a new $\ve_1$,
\beq
\biggl(\sum_{k \leq N} (A_k \psi_k - \xi)\biggr)^2 \leq \sum_{k \leq N} A_k + \xi \sum_{\substack{\kappa, \nu \leq m\\
k \neq \nu}} N_\nu N_\kappa + \ve_1 (N')^2.
\label{eq:7.19}
\eeq
Consequently, by $A_k \geq 1$, $\psi_k \geq \psi$, $N \leq N'$ and \eqref{eq:7.5}, \eqref{eq:7.7},
\eqref{eq:7.12} we have
\beq
(N' \Delta)^2 \leq N' + \xi({N'}^2 - N') + \ve_1 (N')^2,
\label{eq:7.20}
\eeq
\beq
N \leq N' \leq \frac{1 - \xi }{\Delta^2 - \xi - \ve_1}
\label{eq:7.21}
\eeq
in case of Theorem~H.
Similarly, in case of Theorem~I we have
\beq
(N'\Delta)^2 \leq N'+ \ve_1(N')^2,
\label{eq:7.22}
\eeq
\beq
N \leq N' \leq \frac{1}{\Delta^2 - \ve_1}.
\label{eq:7.23}
\eeq
\hfill Q.E.D.

\section{Proof of Theorems J, K and L}
\label{sec:8}

In order to show our weighted density theorem (Theorem~K)
we will use the notation of Section~\ref{sec:7} with the additional quantity
\beq
D' \stackrel{{\rm def}}{=} \sum_{j \leq N} \bigl(A_j (\psi_j - \psi) \bigr)
\geq D \stackrel{{\rm def}}{=} \sum_{j \leq N} (\psi_j - \psi) \geq 0.
\label{eq:8.1}
\eeq
This quantity, completely neglected in the proofs of Theorems~H and I, will be our crucial one in the following.
We will start from \eqref{eq:7.19} to obtain, instead of \eqref{eq:7.20}--\eqref{eq:7.21}:
\beq
(N' \Delta + D')^2 \leq N' + \xi({N'}^2 - N') + \ve_1 (N')^2
\label{eq:8.1masik}
\eeq
\beq
(N')^2(\Delta^2 - \xi - \ve_1) + N'(2\Delta D') \leq N'(1 - \xi)
\label{eq:8.2}
\eeq
from which, by $\xi < \Delta^2$, we obtain
\beq
D \leq D' \leq \frac{1 - \xi}{2\Delta - \ve_1}.
\label{eq:8.3}
\eeq

Suppose now that $\lambda_0$ is given, the $\lambda_j$'s $(1 \leq j \leq N)$ and their number $N$ are unknown
quantities with
\beq
d_j \stackrel{{\rm def}}{=} \Lambda - \lambda_j \geq 0, \quad d_1 \geq d_2 \geq \dots \geq d_N,
\label{eq:8.4}
\eeq
with prescribed conditions
\beq
0 \leq d_j \leq e_j \qquad e_1 \geq e_2 \geq \dots \geq e_N.
\label{eq:8.5}
\eeq

We will suppose that $f$ and $F$ are the functions of Section~\ref{sec:4}
with the parameter $x$ satisfying
\beq
2/x \leq B \Leftrightarrow x \geq 2/B.
\label{eq:8.6}
\eeq

Since by Corollary~2 or by Jutila's density theorem \cite{Jut} we know that the unknown number $N$ is bounded
by some absolute constant $R \in \mathbb Z$, we can suppose that in our extremal problem $N = R$
by the introduction of additional trivial terms with $e_j = 0$ (consequently $d_j = 0$) for $N < j \leq R$.
These new trivial terms do not change the values of $D$ and those of $S$ and $D^*$, defined below
\begin{align}
D^* = D \cdot F(-\lambda_0)
&= \sum^N_{j = 1} \bigl(F(\lambda_j - \lambda_0) - F(\Lambda - \lambda_0)\bigr)
\label{eq:8.7}\\
&= \sum^R_{j = 1} \bigl( F(d_0 - d_j) - F(d_0) \bigr).
\nonumber
\end{align}
Then, under the constraint $D \leq \widetilde C'$ we are looking for an upper bound for the quantity
\beq
S = \sum^R_{j = 1} (e^{Bd_j} - e^{Cd_j}) = \sum^N_{j = 1} (e^{Bd_j} - e^{Cd_j}),
\label{eq:8.8}
\eeq
with the side constraints \eqref{eq:8.4}--\eqref{eq:8.5} and $B > C \geq 0$, where $R$ is now
a fixed, large constant.
The upper bound will naturally depend on $B$, $C$ and $C'$ but not on~$R$.

%%%innen ujPINTZ4.texbol

Let
\beq
\gathered
T = t_0(f), \ \ b = \frac{B}{T} \geq 1, \ \ 0 \leq c = \frac{C}{T} < b,\\
Y_j = e^{e_j T} \geq y_j = e^{d_j T} \geq 1, \ \  h_1(y) = y^v, \ \ h_2(y) = y^b - y^c.
\endgathered
\label{eq:8.9}
\eeq
Then we have with the above notation
\beq
D^*_1(y) := \sum^R_{j = 1} F(d_0 - d_j) = T \int\limits^1_0 f(vT) e^{-d_0 vT} \sum^R_{j = 1} y^v_j \, dv.
\label{eq:8.10}
\eeq

The following observation is sufficient to show Theorem~K.

\medskip
\noindent
{\bf Proposition.} {\it If $y \geq z > 1$, $\eta > 0$, $0 < v < 1$, $b \geq 1$, $0 \leq c < b$, then
\beq
H_i(y, z, \eta) = h_i (y + \eta) + h_i(z - \eta) - \bigl(h_i(y) + h_i(z)\bigr) \aligned &<\\
&>\endaligned\ 0 \text{ for } \aligned &i = 1,\\
&i = 2.\endaligned
\label{eq:8.11}
\eeq
}

\noindent
{\bf Proof.} $h'_1(y)$ is decreasing, $h'_2(y)$ is increasing for $y \geq 1$ due to
\beq
h''_1(y) = v(v - 1) y^{v - 2} < 0, \quad
h''_2(y) = b(b - 1)y^{b - 2} - c(c - 1)y^{c - 2} > 0.
\label{eq:8.12}
\eeq
Consequently,
\beq
H_i(y, z, \eta) = \int\limits^\eta_0 \bigl(h'_i(y + t) - h'_i(z - \eta + t)\bigr) dt \aligned &< \\
&> \endaligned\ 0 \text{ for } \aligned &i = 1,\\
&i = 2.\endaligned
\label{eq:8.13}
\eeq

The proposition means that if we have a given configuration of the variables
$\{y_j\}^R_{j = 1}$ with $y_i \geq y_{i + 1}$, $Y_i \geq Y_{i + 1}$, $Y_i \geq y_i \geq 1$,
then this configuration cannot yield a maximum for the $h^*_2(\underline y) = \sum\limits^R_{j = 1} h_2(y_j)$ if there is a possibility to increase the distance between two variables among $y_1, \dots, y_R$.
According to this, let $r$ be the largest index with $y_r > 1$ in the maximal system $\{y_i\}^R_{i = 1}$.
Then necessarily
\beq
y_i = Y_i \Leftrightarrow d_i = e_i \quad \text{ for } \ i = 1, \dots, r - 1.
\label{eq:8.14}
\eeq
Namely, otherwise we could change with a small $\eta > 0$ $\,y_k$ to $y_k + \eta$, $y_r$ to $y_r - \eta$ and obtain a larger value for $h^*_2(\underline y)$ if $k$ is defined by
\beq
k = \min \{\nu; \ y_\nu < Y_\nu\},
\label{eq:8.15}
\eeq
while the corresponding function $h^*_1(\underline y) = \sum\limits^R_{j = 1} y^v_j$, and consequently $D^*_1(\underline y)$ would decrease and thus $D^* \leq \widetilde D$ would still hold for the new system
$\underline y$.
This proves Theorem~K.

In order to show Theorem~L, taking into account Remark~5, suppose that the first index, for which in the maximum case we do not have equality in \eqref{eq:4.47} is $k \in [1, M]$.
The case $k = M$ is clearly impossible since then we could increase $y_k$ in view of $d_k < d_{k - 1}$ which follows by \eqref{eq:4.48} from
\beq
F(d_0 - d_k)-F(d_0) < c(k) - c(k - 1) \leq c(k - 1) - c(k - 2) = F(d_0 - d_{k - 1})-F(d_0).
\label{eq:8.16}
\eeq
If we increase $y_k$ that would lead to an increase of $h_2(y) = y_k^b - y_k^c$ and thereby to an increase of $S^*$.

If $k < M$ we also must have $d_k < d_{k - 1}$ by \eqref{eq:8.16}.
Suppose that we have exactly $\ell \geq 1$ equal variables after $k$, that is we have
\beq
d_k \geq d = d_{k + 1} = \dots = d_{k + \ell} > d_{k + \ell + 1} \geq 0
\label{eq:8.17}
\eeq
and $d_{k + \ell}$ is not the last term.
We clearly have $d = d_{k + \ell} > 0$ if it is the last term, otherwise we could simply slightly increase $d_k$, that is, increase $y_k$ slightly which would yield a larger value for $S^*$.
If $\ell = 1$ we can substitute $y_{k + 1}$ by $y_{k + 1} -\eta$, $y_k$ by $y_k + \eta$ with a sufficiently small $\eta$ and we obtain a contradiction.
If $\ell \geq 2$ then we cannot have equality in any of the $\ell - 1$ relations of type \eqref{eq:4.47} for $m = k + i$, $1 \leq i \leq \ell - 1$, since if the first one for which \eqref{eq:4.47} is sharp has index $m$, then
\beq
c(m) - c(m - 1) < F(d_0 - d_m) - F(d_0) = F(d_0 - d_{m + 1}) - F(d_0) \leq c(m  + 1) - c(m).
\label{eq:8.18}
\eeq
which contradicts \eqref{eq:4.48}.
But then we can substitute similarly to the case $\ell = 1$ $\, y_k$ by $y_k + \eta$, $y_{k + \ell}$ by $y_{k + \ell} - \eta$ with a sufficiently small $\eta$ and we again arrive at a contradiction. This proves Theorem~L.

\section{Proof of Theorem~\ref{th:1}}
\label{sec:9}
According to \eqref{eq:2.36}--\eqref{eq:2.46uj}
our task will be to show with some small
but fixed constant $c_0 > 0$
\beq
S_0 = \sum^M_{i = 1} S^2_i \leq 1 - c_0.
\label{eq:9.1}
\eeq
Let us dissect the sums $S_i$ as
\beq
S_i = \frac{25}{7} \int\limits^H_0 N_i(\lambda) e^{-\frac{25}{7}\lambda} d_\lambda = \frac{25}{7}\int\limits^{\Lambda_0}_0 + \frac{25}{7}
\int\limits^H_{\Lambda_0} = a_i + b_i,
\label{eq:9.2}
\eeq
where $\Lambda_0 = 1.311$,
\beq
N_i(\lambda) = \sum_{\substack{\varrho_j \in R, \ \lambda_j \leq \lambda\\
\chi_j \sim \chi_i}} 1.
\label{eq:9.3}
\eeq
Our basic inequality will be a small refinement of
\beq
S \leq \sum_{i \leq M} a^2_i +  \max_{i \leq M} b_i  \biggl( 2 \sum_{i \leq M} a_i
+ \sum_{i \leq M} b_i\biggr),
\label{eq:9.4}
\eeq
where we will treat two classes (and eventually its conjugate classes) containing the zeros with the greatest real part separately. According to \eqref{eq:9.4} we will estimate $\sum a_i$, $\sum b_i$, $\max b_i$, using Principles~1--3
in form of Theorems~C--K and a few other results of \cite{Hea} and \cite{Xyl}.

%%%ExplicitII.1.o.
As mentioned already in the introduction, we will try to give a relatively simple proof leaving many possibilities for improvement for future parts of this series.

Throughout we will use the notation that the classes will be ordered according to decreasing value of the greatest real part of the zeros belonging to the relevant class, so according to increasing value of $\lambda_i = \lambda_{i1}$ where the other zeros of the same class will be ordered as $\lambda_{i1} \leq \lambda_{i2} \leq \dots$.
Zeros will be ordered and counted always by multiplicity.
In contrast to \cite{Hea} and \cite{Xyl} we will include also conjugate classes and conjugate zeros in the calculation.
We will distinguish first

Case I. \ \ \ \ \ \ \ $\lambda_1 > 0.44$

Case II/A. \ \ $0.35 < \lambda_1 \leq 0.44$

Case II/B. \ \ $\lambda_1 \leq 0.35$

According to Theorem~E of \cite{Xyl} we have in Case II at most the real zero $\varrho_1 = 1 - \delta = 1 - \lambda_1 \mathcal L$ of the real non-principal $\chi_1$ with the property $\lambda_{ij} \leq 0.44$.
The reason to distinguish between Cases II/A and II/B is that in Case II/B we have no other zeros with $\lambda \leq \Lambda_0$ (in fact with $\lambda \leq 1.42$) while in Case II/A we might have zeros with $\lambda > 1.18$ (see Table 7 on p.~301 of \cite{Hea}).

We will begin the estimation of $\max b_i$: Important role will be played by Lemma 10.3 of \cite{Hea}, p.~316, according to which apart from at most two characters and its conjugates we will have $\lambda_i \geq \frac67 - \ve$ for $q > q_0(\ve)$ for each character.
Since we have by Theorem~F apart from at most two zeros $\lambda_i \geq 0.702$, we will distinguish the following cases for the estimation of $\max b_i$

\noindent
Case 1\qquad $\lambda_i \geq 6/7 - \ve$ (surely valid for $i > 4$)\\
Case 2\qquad $\lambda_i \geq 0.702$ \ \ \ (surely valid for $i > 2$)\\
Case 3\qquad $0.35 < \lambda_i < 0.702$

In this case $i = 1$ or $2$ and $\chi_i = \chi_1$ or $\overline\chi_1$.\\
Case 4\qquad $\lambda_i \leq 0.35$

In this case $\chi_1$ and $\varrho_1$ are real and $\lambda > 1.42 > \Lambda_0$ for all other zeros \cite[Tables 3 and 7]{Hea}.

\smallskip
In order to calculate an upper estimate for $b_i$ we will apply for Cases 1--4 Theorem~I with $\lambda_0 = \frac67 - 10^{-8}$, $0.702$, $0.35$ and $0$, resp.\ and in the last case we will take into account $\lambda > 1.42$ for all other zeros.
Using Theorem~I we can give a lower estimate for the first few zeros with $\lambda_{ij} \leq 3$ (their number is in Cases 1--4 at most $45$, $38$, $34$ and $31$, resp.) and then apply an upper estimate for the zeros below $3 + k/10$ $(k = 0, 1, 2, \dots)$ until about $6$ which is approximately the limit for Theorem~I.
We will actually use Theorem~I until $\Lambda_{2,1} = 6.6$, $\Lambda_{2.2} = 6.4$,
$\Lambda_{2,3} = 6$ and $\Lambda_{3,4} = 5.8$ in Cases 1--4, resp.
Further we will use in Cases 1--4 the values $\lambda_0 = 6/7 -10^{-8}$, $\lambda_0 = 0.702$, $\lambda_0 = 0.35$, $\lambda_0 = 0$, resp.
(The limit of Theorem~I will be larger if $\lambda_0$ is larger.)
On the other hand the value $x$ for $F_x(z) = G\left(\frac{z}{x}\right)$ (see \eqref{eq:4.24}) is chosen experimentally to obtain the approximately optimal estimate for the $N^{\text{\rm th}}$ zero of the same class or to bound $N_i(3 + k/10)$ for $3 \leq 3 + k/10 \leq \Lambda_{2\nu}$ $(1 \leq \nu \leq 4)$.
The condition $B = 25/7 > t_0(f) = 2/x$ will be always satisfied as well as $\lambda_0/x \leq 5/4$ which assures Condition~2 (see \eqref{eq:4.26}) for $F_x(z) = G(z/x)$.
For all the other zeros of the same class, i.e.\ for $\Lambda_{2\mu} \leq \lambda_{ij} \leq \Lambda_\infty = \log\log\log q$ $(1 \leq \mu \leq 4)$ we can use our estimate \eqref{eq:4.11} of Corollary~2 of Theorem~D.
For simplicity we can calculate in all Cases 1--4 with $E_3$ arising from $\Lambda_3 = 5.8$ valid for all $\Lambda \geq 5.8$.
We obtain

\begin{Lem}
\label{lem:6uj}
We have $\max b_i \leq c_j^*$ in Case~j, where
\beq
\label{eq:9.5}
c_1^* = 0.0722, \ \ c_2^* = 0.0751, \ \ c_3^* = 0.0826, \ \ c_4^* = 0.715.
\eeq
\end{Lem}

\begin{proof}
We give just a brief account of the results of the calculation for the typical Case~1 (which applies apart from at most four classes for all others).
We obtain at most $6$ zeros below $\Lambda_0 = 1.311$ for which the corresponding value $e^{-25/7 \max(\lambda_{ij}, \Lambda_0)}$ is independently from the concrete value $\lambda_{ij} \leq \Lambda_0$ just $e^{-(25/7)\Lambda_0}$.
For the other possible zeros below $3$ we get the bounds (in brackets the value of the parameters $x$ used for the function $F_x(z) = G(z/x)$)
$\lambda_{i7} \geq 1.47$ $(1.58)$, $\lambda_{i8} \geq 1.61$ $(1.6)$, $\lambda_{i9} \geq 1.73$ $(1.62)$, $\lambda_{i10} \geq 1.85$ $(1.66)$, $\lambda_{i11} \geq 1.94$ $(1.66)$, $\lambda_{i12} \geq 2.05$ $(1.68)$, $\lambda_{i13} \geq 2.12$ $(1.68)$,
$\lambda_{i14} \geq 2.20$ $(1.68)$, $\lambda_{i15} \geq 2.27$ $(1.68)$, $\lambda_{i16} \geq 2.33$ $(1.68)$, $\lambda_{i17} \geq 2.4$, $\lambda_{i18} \geq 2.45$, $\lambda_{i19} \geq 2.51$, $\lambda_{i20} \geq 2.56$, $\lambda_{i21} \geq 2.61$, $\lambda_{i22} \geq  2.65$,
$\lambda_{i23} \geq 2.7$, $\lambda_{i24} \geq 2.74$, $\lambda_{i25} \geq 2.78$, $\lambda_{i26} \geq 2.82$, $\lambda_{i27} \geq 2.85$,
$\lambda_{i28} \geq 2.89$, $\lambda_{i29} \geq 2.92$, $\lambda_{i30} \geq 2.95$, $\lambda_{i31} \geq 2.99$ $(\lambda_{i32} \geq 3)$ with the parameters $x_{17} = x_{18} = \dots = x_{31} = 1$.
Similarly we can calculate with experimentally optimally chosen parameters $x = x_k' \in [0.6, 1.7]$ an upper estimate for $N_i(3 + k/10)$ for $0 \leq k \leq 35$.
\end{proof}

%%1--4oldalvege

The value $\sum b_i$ can be easily estimated by Corollary~1 as
\begin{align}
\sum_{i \leq I} b_i
&\stackrel{{\rm def}}{=} \frac{25}{7} \int\limits^H_{\Lambda_0} N(\lambda) e^{-\frac{25}{7}\lambda} d\lambda
\leq \sum_{\lambda_j \leq H} e^{-\frac{25}{7}\max(\lambda_j, \Lambda_0)}
\label{eq:9.6}\\
&\leq e^{-19 \Lambda_0 /21} \sum_{\lambda_j \leq H} e^{-(8/3) \max(\lambda_j, \Lambda_0)}
\leq e^{-19\Lambda_0/21} \sum_{\lambda_j \leq H} e^{-\frac83 \lambda_j} e^{-\frac{r + \kappa}{2} d_j}\nonumber\\
 &< 22.281 e^{-19\Lambda_0/21} < 6.805.\nonumber
\end{align}

In order to estimate $\sum a_i$ we will distinguish 8 cases as follows ($h$ is a small constant)

\begin{tabular}{l@{\hspace*{25mm}}l}
Case 1 \qquad $\lambda_1 \geq 0.68$ & Case 5 \qquad $0.35 \leq \lambda_1 < 0.44$ \\[2pt]
Case 2 \qquad $0.6 \leq \lambda_1 < 0.68$ & Case 6 \qquad $0.14 \leq \lambda_1 < 0.35$\\[2pt]
Case 3 \qquad $0.5 \leq \lambda_1 < 0.6$ & Case 7 \qquad $0.04 \leq \lambda_1 < 0.14$\\[2pt]
Case 4 \qquad $0.44 \leq \lambda_1 < 0.5$ & Case 8 \qquad $\lambda_1 < 0.06$
\end{tabular}

\medskip
\vskip12pt
{%%\baselineskip=36pt
In the most sophisticated Cases 1--5 we will use Theorem~K (for its proof see Section~\ref{sec:8})
with the parameters $\lambda_0 = 0.44$, $x = 0.68$ for Cases 2--5 and $x = 0.7$ for Case~1, $\varphi = \frac13$.
We will choose $\Lambda_0$ in such a way that it should be just slightly smaller than the value $\lambda$ for which
\beq
\left(G\left(\frac{\lambda - \lambda_0}{x} \right) - \frac{f(0) \varphi}{2} \right)^2 =
\frac{f(0) \varphi}{2} G \left( - \frac{\lambda_0}{x}\right) \Leftrightarrow \psi = \xi + \sqrt\xi
\label{eq:9.7}
\eeq
holds.
$\Lambda_0 = 1.311$ will be such a choice.

In view of $f(0) = 16x/15$ we have then with the notation of \eqref{eq:7.4}--\eqref{eq:7.5}
\beq
\gathered
G \left(-\frac{\lambda_0}{x} \right) = 1.56903 \dots, \quad
\frac{f(0)\varphi}{2} = \frac8{45} \cdot 0.7, \quad \lambda = \Lambda_0 = 1.311, \\
G\left(\frac{\Lambda - \lambda_0}{x}\right) = 0.5882\dots \\
\psi = 0.37488\dots, \quad \xi = 0.07931\dots \quad \Delta = 0.29557\dots
\endgathered
\label{eq:9.8}
\eeq
and consequently by \eqref{eq:4.36}
\beq
D = \sum(\psi_j - \psi) \leq 1.5575, \
D_0 = \sum \left(\! G \! \left(\! \frac{\lambda_j\! -\! \lambda_0}{x}\!\right)\! - \!
G \left(\! \frac{\lambda\! -\! \lambda_0}{x} \! \right)\!\right) \leq 2.4438.
\label{eq:9.9}
\eeq

According to the theorem the sum
\beq
S' = \sum_{\lambda_j \leq \Lambda_0} \bigl(e^{-\frac{25}{7}\lambda_j} - e^{-\frac{25}{7}\Lambda_0}\bigr) = e^{-\frac{25}{7}\Lambda_0}
\sum_{\lambda_j \leq \Lambda_0} (e^{\frac{25}{7}d_j} - 1) = e^{-\frac{25}{7}\Lambda_0} S
\label{eq:9.7masik}
\eeq
is, in view of Theorem~K maximal, if, taking into account Theorems E, F, we choose
\beq
\lambda_1 = \lambda_2 = 0.68, \quad \lambda_3 = \dots = \lambda_k = 0.702
\label{eq:9.8masik}
\eeq
and $\lambda_{k + 1} \in [0.702, \Lambda_0]$ in such a way that $D_0 = 2.4438$ should hold.
Since we have
\beq
G\left(\frac{\lambda_1 \! -\! \lambda_0}{x} \right) = \frac89, \
G\left(\frac{\lambda_3 \! -\! \lambda_0}{x} \right) = 0.8747\dots, \
G \left(\frac{\Lambda_0 \! -\!  \lambda_0}{x} \right) = 0.5882\dots
\label{eq:9.9masik}
\eeq
we obtain $k = 6$, $G \left(\frac{\lambda_9 - \lambda_0}{x} \right) = 0.71336\dots$, $\lambda_9 =
0.99\dots$.
Consequently,
\beq
S' \leq 2e^{-(25/7)\cdot 0.68}  + 6e^{-(25/7)\cdot 0.702} + e^{-(25/7)\cdot 0.99} - 9e^{-(25/7) \Lambda_0} < 0.612
\label{eq:9.10}
\eeq
which settles Case 1.\par}

Similarly we obtain the result for Cases 1--5, that is,

\begin{Lem}
\label{lem:7uj}
We have
\beq
\label{eq:9.14}
\sum_{i \leq I} a_i \leq \widetilde c_\nu \ \text{ for Case } \nu \ (1 \leq \nu \leq 5) \ \text{ above, where }
\eeq
\beq
\label{eq:9.15}
\widetilde c_1 = 0.612, \ \widetilde c_2 = 0.622, \ \widetilde c_3 = 0.564, \ \widetilde c_4 = 0.453, \ \widetilde c_5 = 0.483.
\eeq
\end{Lem}

\begin{proof}
The proof is completely analogous in Cases 2-3 where we use that apart from two zeros which might be as large as the lower bounds stipulated in Case~$\nu$ we have for all other zeros, i.e.\ for $\lambda_3$, by Tabellen 2, 3, 7 of \cite{Xyl}
\beq
\label{eq:9.16}
\lambda \geq c_\nu', \ \ c_2' = 0.74, \ \ c_3' = 0.97.
\eeq

In Case~4, if $\chi_1$ or $\varrho_1$ is complex, then by Tabelle~7 of \cite{Xyl} we have at most two zeros, $\varrho_1$ and $\overline \varrho_1$ (of $L(s, \chi_1)$ or $L(s, \chi_1)$ and $L(s, \overline \chi_1)$, respectively) with $\lambda \leq \Lambda_0$ (in fact if $\lambda \neq \lambda_1$ then $\lambda \geq 1.36$)
\beq
\label{eq:9.17}
\sum_{i\leq I} a_i \leq 2\bigl(e^{-(25/7)\cdot 0.44} - e^{-(25/7)\Lambda_0}\bigr) < 0.39698.
\eeq

If Case 4 holds and  $\chi_1$ and $\varrho_1$ are real, then by Tables 4 and 7 of \cite{Hea} we have apart from this \emph{single} zero $\lambda \geq 1.08$ for all other zeros, so we can apply the same procedure as in Case~1 (cf.\ \eqref{eq:9.7}--\eqref{eq:9.14}) and obtain \eqref{eq:9.14} in this case with an upper bound.
Comparison with \eqref{eq:9.17} yields the estimate \eqref{eq:9.15} for $\nu = 4$.
\end{proof}

Case 5 is more simple in the sense that in this case $\chi_1$ and $\varrho_1$ must be real by Theorem~E.
Further we have by Tables 4 and 7 of \cite{Hea} $\lambda \geq 1.18$ apart from this single zero.
Applying again the same procedure as before (cf.\ \eqref{eq:9.7}--\eqref{eq:9.10}) we obtain \eqref{eq:9.14}--\eqref{eq:9.15} for $\nu = 5$.

Case 6 is even more simple since in this case we have just the single real $\varrho_1$ for real $\chi_1$ within $R(\Lambda_0, T)$.
This means that in this case we have
\beq
\label{eq:9.18}
\sum_{i \leq I} a_i = a_1 \leq \bigl(e^{-(25/7)\cdot 0.14} - e^{-(25/7)\Lambda_0}\bigr) = 0.59727\ldots.
\eeq
The same applies in Cases 7 and 8 when
\beq
\label{eq:9.19}
\sum_{i \leq I} a_i = e^{-(25/7)\lambda_1} - e^{-(25/7)\Lambda_0} .
\eeq

Summarizing we have

\begin{Corol}
In Cases 7--8 we have \eqref{eq:9.19} while in Cases 1--6
\beq
\label{eq:9.20}
\sum_i a_i \leq 0.622.
\eeq
\end{Corol}

Cases 7--8 we will settle just using the results of \cite{Hea} about further zeros by the aid of Theorems C and D.
For Cases 1--6 we state

\begin{Corol}
In Cases 1--6 we have for $S$ in \eqref{eq:9.4} the estimate
\beq
\label{eq:9.21}
S < 0.9903.
\eeq
\end{Corol}

\begin{proof}
Let us forget for a moment that the typical estimate $c_1^* = 0.0722$ holds up to at most four exceptional classes for $\max b_i$.
If we had no exceptions then in Cases 1--6 we would have \eqref{eq:9.20} and this would lead by \eqref{eq:9.6} and \eqref{eq:9.4} to
\beq
\label{eq:9.22}
S^- \leq 0.622^2 + 0.0722 (2 \cdot 0.622 + 6.805) = 0.9680218.
\eeq
However, two of the classes might have a surplus
\beq
\label{eq:9.23}
c_3^* - c_1^* = 0.0826 - 0.0722 = 0.0104
\eeq
for $\max b_i$ and this surplus obtains the factor at most $1.244 + 2\cdot 0.0826 = 1.4092$ from $2 \sum a_i$ and from $b_{i_1} + b_{i_2}$ (the corresponding two exceptional classes).
This leads to the surplus
\beq
\label{eq:9.24}
\Delta_1 = 0.01465568.
\eeq
Analogously we might have another smaller surplus
\beq
\label{eq:9.25}
c_3^* - c_1^* = 0.0751 - 0.0722 = 0.0029
\eeq
for $\max b_i$ with a factor $2\cdot 0.0751 = 0.1502$ (since we calculated already the contribution of $\sum a_i$ with the larger surplus for all classes).
This yields another surplus of size
\beq
\label{eq:9.26}
\Delta_2 = 0.0029 \cdot 0.1502 = 4.3558 \cdot 10^{-4}.
\eeq

Adding $\Delta_1 + \Delta_2$ to $S^-$ in \eqref{eq:9.22} we obtain $S < 0.9832$, i.e. \eqref{eq:9.21} holds for Cases 1--6.
\end{proof}

In Case 7 we have by Tables 4 and 5 of \cite{Hea} apart from the single real zero $\varrho_1$ for all other zeros $\lambda \geq 2.421 = \Lambda_1$ and we define $a_i$, $b_i$ by $\Lambda_1$ instead of $\Lambda_0$.

Consequently we have by \eqref{eq:9.4}--\eqref{eq:9.5}, \eqref{eq:9.19} and similary to \eqref{eq:9.6} in this case
\beq
\label{eq:9.27}
\sum_i b_i \leq 15.6 e^{-19\Lambda_1/21} < 1.74516,
\eeq
\beq
\label{eq:9.28}
\max b_i < c_4^* = 0.0715,
\eeq
\beq
\label{eq:9.29}
\sum a_i < e^{-(25/7)0.04} - e^{-(25/7)\Lambda_1} < 0.86671,
\eeq
\beq
\label{eq:9.30}
S \leq 0.86671^2 + 0.0715(2\cdot 0.86671 + 1.74516) < 0.99991.
\eeq

We note that although the estimate \eqref{eq:9.28} was shown for the value $\Lambda_0$ instead of $\Lambda_1$, but the definition
\[
b_i = b_i(\Lambda) = \frac{25}{7} \sum\limits_{\Lambda}^H N_i (\lambda) e^{-(25/7)\lambda} d\lambda
\]
is clearly decreasing in $\Lambda$ so in fact we would get a much better estimate for $c_4^*$ with $\Lambda_1$ in place of $\Lambda_0$.

Finally, in Case 8 we have again a single real zero $\varrho_1$ with $\lambda_1 \leq 0.04$, while for all other zeros we have by Theorem~G and Table~5 of \cite{Hea}
\beq
\label{eq:9.31}
\lambda \geq \Lambda^* = \max \left(\Bigl(\frac{12}{11} - \ve\Bigr) \log \frac1{\lambda_1}, \Lambda_2 \right)
\ \text{ with } \ \Lambda_2 = 3.96.
\eeq
Consequently, we have, by Corollary~1 for $\lambda \geq \Lambda^*$
\beq
\label{eq:9.32}
\sum_{\Lambda^* \leq \lambda_j \leq \Lambda_\infty}  e^{-(8/3)\lambda_j} < 10.4, \ \ \
\sum_{\substack{\lambda_{ij} \in \kappa\\
\Lambda^* \leq \lambda_{ij} \leq \Lambda_\infty}} e^{-2\lambda_{ij}} < 10.4.
\eeq
This implies (defining $b_i$ now with $\Lambda^* = \Lambda^*(\lambda_1)$)
\beq
\label{eq:9.33}
\sum_{\Lambda^* \leq \lambda_j \leq \Lambda_\infty} b_i \leq e^{-(\frac{25}{7} - \frac83)\Lambda^*} \cdot 10.4 \leq 10.4\lambda_1^{76/77} < 0.435,
\eeq
\beq
\label{eq:9.34}
\max b_i \leq e^{-\left(\frac{25}{7} - 2\right)\Lambda^*} \cdot 10.4 \leq 10.4 \lambda_1^{11/7}
\eeq
\beq
\label{eq:9.35}
\sum a_i = a_1 = e^{-(25/7)\lambda_1} - e^{-(25/7)\Lambda^*} < e^{-\frac{25}{7} \lambda_1}.
\eeq
So we have by \eqref{eq:9.4} and $\lambda_1 \leq 0.04$
\beq
\label{eq:9.36}
S \leq e^{-\frac{50}{7} \lambda_1} + 2.87 \cdot 10.4\lambda_1^{11/7} \leq e^{-7\lambda_1} + 5\lambda_1 < 1,
\eeq
since $(1 - e^{-y})/y$ is decreasing for $y \geq 0$ and so we have for $\lambda_1 \leq 0.04$
\beq
\label{eq:9.37}
\frac{1 - e^{-7\lambda_1}}{\lambda_1} \geq \frac{1 - e^{-0.28}}{0.04} > 6.1 > 5.
\eeq
Thereby \eqref{eq:9.36} is really true which settles the remaining Case 8 and consequently the proof of Theorem~\ref{th:1} is complete.

\medskip
\noindent
{\bf Acknowledgement}.
The author would like to thank his colleague and friend, Sz. Gy. R\'ev\'esz, for supplying the proof of Lemma~\ref{lem:5}.

%%%%eddig

\noindent
J\'anos Pintz\\
R\'enyi Mathematical Institute\\
of the Hungarian Academy of Sciences\\
Budapest, Re\'altanoda u. 13--15\\
H-1053 Hungary\\
e-mail: pintz.janos@renyi.mta.hu

\end{document}